\documentclass{aic}

\usepackage[abbrev,msc-links,backrefs]{amsrefs} 
\usepackage{doi}

\renewcommand{\PrintDOI}[1]{\doi{#1}}

\usepackage{enumitem}
\usepackage{tikz}
\usepackage{amsmath,amssymb,amsthm}
\usepackage{thmtools, thm-restate}

\theoremstyle{plain}
\newtheorem{thm}{Theorem}[section]
\newtheorem{theorem}[thm]{Theorem}
\newtheorem{lemma}[thm]{Lemma}
\newtheorem{corollary}[thm]{Corollary}
\newtheorem{proposition}[thm]{Proposition}
\newtheorem{conjecture}[thm]{Conjecture}

\newtheorem*{claim*}{Claim}

\theoremstyle{definition}
\newtheorem{definition}[thm]{Definition}
\newtheorem{remark}[thm]{Remark}

\newcommand{\abs}[1]{\ensuremath{{\lvert {#1} \rvert}}}

\DeclareMathOperator{\field}{GF}
\DeclareMathOperator{\bw}{bw}

\DeclareMathOperator{\n}{N}

\aicAUTHORdetails{%
  title = {Obstructions for Bounded Branch-depth in Matroids}, 
  author = {J.~Pascal Gollin, Kevin Hendrey, Dillon Mayhew, and Sang{-}il Oum},
  plaintextauthor = {J.~Pascal Gollin, Kevin Hendrey, Dillon Mayhew, and Sang{-}il Oum},
  keywords = {matroids, branch-depth, branch-width, twisted matroids, representable matroids, quasi-graphic matroids},
}   

\aicEDITORdetails{%
   year={2021},
   number={4},
   received={1 April 2020},   
   revised={22 January 2021},    
   published={24 May 2021},  
   doi={10.19086/aic.24227},      
}   

\begin{document}

\begin{frontmatter}[classification=text]

\title{Obstructions for Bounded Branch-depth\\ in Matroids} 

\author[jpg]{J.~Pascal Gollin\thanks{Supported by the Institute for Basic Science (IBS-R029-C1).}}
\author[kh]{Kevin Hendrey\footnotemark[1]}
\author[dm]{Dillon Mayhew\thanks{Supported by a Rutherford Discovery Fellowship, managed by the Royal Society Te Ap\={a}rangi.}}
\author[so]{Sang{-}il Oum\footnotemark[1]}

\begin{abstract}
    DeVos, Kwon, and Oum introduced the concept of branch-depth of matroids
    as a natural analogue of tree-depth of graphs. 
    They conjectured that a matroid of sufficiently large branch-depth contains the uniform matroid~$U_{n,2n}$ or the cycle matroid of a large fan graph as a minor. 
    We prove that matroids with sufficiently large branch-depth 
    either contain the cycle matroid of a large fan graph as a minor
    or have large branch-width.
    As a corollary, we prove their conjecture for matroids representable over a fixed finite field and quasi-graphic matroids, where the uniform matroid is not an option. 
\end{abstract}
\end{frontmatter}

\section{Introduction}
\label{sec:intro}

Motivated by the notion of tree-depth of graphs, 
DeVos, Kwon, and Oum~\cite{branchdepth} introduced the branch-depth\footnotemark\ 
of a matroid~$M$ as follows. 
\footnotetext{Note that there is also a different notion of branch-depth of a matroid introduced by Kardo\v{s}, Kr\'{a}l', Liebenau, and Mach~\cite{KKLM-depth}. 
The differences between these notions are discussed in~\cite{branchdepth}.}
Recall that the connectivity function~$\lambda_M$ of a matroid~$M$ is defined as ${\lambda_{M}(X) = r(X) + r(E(M) \setminus X) - r(E(M))}$, 
where~$r$ is the rank function of~$M$.
A \emph{decomposition} is a pair~${(T,\sigma)}$ consisting of a tree~$T$
with at least one internal node and a bijection~$\sigma$ from~${E(M)}$ to the set of leaves of~$T$. 
For an internal node~$v$ of~$T$, the \emph{width} of~$v$ is defined as 
\[
    \max_{\mathcal{P}' \subseteq \mathcal{P}_{v}} \lambda_{M} \left( \bigcup_{X \in \mathcal{P}'} X \right), 
\]
where~${\mathcal{P}_v}$ is the partition of~${E(M)}$ into sets induced by components of~${T - v}$ under~$\sigma^{-1}$. 
The \emph{width} of a decomposition~${(T,\sigma)}$ is defined as the maximum width of its internal nodes. 
The radius of~${(T,\sigma)}$ is the radius of~$T$. 
A decomposition is a \emph{$(k,r)$-decomposition} if its width is at most~$k$ and its radius is at most~$r$. 
The \emph{branch-depth} of a matroid~$M$ is defined 
to be the minimum integer~$k$ 
for which~$M$ admits a ${(k,k)}$-decomposition
if~${E(M)}$ has more than one element,
and is defined to be~$0$ otherwise. 

It is well known that graphs of large tree-depth contain a long path as a subgraph (see the book of Ne\v{s}et\v{r}il and Ossona de Mendez~\cite{NO2012}*{Proposition 6.1}). 
DeVos, Kwon, and Oum~\cite{branchdepth} made an analogous conjecture for matroid branch-depth as follows. 
Since the cycle matroid of a path graph has branch-depth at most~$1$, paths no longer are obstructions for small branch-depth. 
Instead, they use the cycle matroid of fans. 

The \emph{fan matroid}~${M(F_n)}$ is the cycle matroid of a fan~$F_n$, which is the union of a star~$K_{1,n}$ together with a path with~$n$ vertices through the leaves of the star, see Figure~\ref{fig:fan}. 
Note that the path with~${2n - 1}$ vertices is a fundamental graph of~${M(F_n)}$ (we will define the fundamental graph of a matroid with respect to a base in Section~\ref{sec:twisted}). 

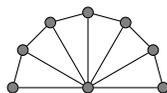
\begin{figure}[h]
    \centering
    \begin{tikzpicture}
        \tikzstyle{every node}=[circle,draw,fill=black!50,inner sep=0pt,minimum width=4pt]
        \node at(0,0)(v){};
        \draw (0:1) node (x0){} 
        \foreach \x in {1,...,6}{
            -- (30*\x:1) node (x\x){}
        };
        \foreach \x in {0,...,6}
            \draw (x\x)--(v);
    \end{tikzpicture}
    \caption{The fan~$F_7$. The fan matroid~${M(F_7)}$ is the cycle matroid of a fan~$F_7$.}\label{fig:fan}
\end{figure}

DeVos, Kwon, and Oum~\cite{branchdepth}*{Propositions~4.1,~4.2, and Lemma~5.17} showed that the fan matroid~${M(F_n)}$ has branch-depth~${\Theta(\log n/\log\log n)}$. 
Hence, fan matroids are indeed obstructions for bounded branch-depth. 

We write~$U_{n,2n}$ to denote the uniform matroid of rank~$n$ on~${2n}$ elements. 
Now, here is the conjecture. 

\begin{conjecture}[DeVos, Kwon, and Oum~\cite{branchdepth}]
    \label{conj:branchdepth}
    For every positive integer~$n$,
    there is an integer~$d$
    such that every matroid of branch-depth at least~$d$ 
    contains a minor isomorphic to~${M(F_n)}$ or~${U_{n,2n}}$. 
\end{conjecture}

Our main theorem verifies their conjecture for matroids of bounded branch-width as follows.
Note that~$U_{n,2n}$ has large branch-width if~$n$ is big and so there is no need for the following theorem to mention~$U_{n,2n}$. 

\begin{restatable}{theorem}{mainthmintro}
    \label{thm:main-intro}
    For positive integers~$n$ and~$w$, 
    every matroid of branch-depth at least~${3(3w)^{2n}}$ 
    contains a minor isomorphic to~${M(F_n)}$ or has branch-width more than~$w$. 
\end{restatable}

This allows us to obtain the following corollary for matroids representable over a fixed finite field, since we can use a well-known grid theorem for matroids of high branch-width by Geelen, Gerards, and Whittle~\cite{GGW:large-branchwidth}.

\begin{restatable}{corollary}{mainthmintrorepresentable}\label{cor:main-intro-representable}
    For every~positive integer $n$ and every finite field~$\field(q)$, 
    there is an integer~$d$
    such that every $\field(q)$-representable matroid with branch-depth at least~$d$ 
    contains a minor isomorphic to~${M(F_n)}$. 
\end{restatable}
Previously, Kwon, McCarty, Oum, and Wollan~\cite{obstructions}*{Corollary~4.9} verified the conjecture for binary matroids, 
as a corollary of their main result about vertex-minors and rank-depth. 

In a big picture, our proof follows the strategy of Kwon, McCarty, Oum, and Wollan~\cite{obstructions}. 
As the branch-width is small, we can find, in every large set, a large subset having small connectivity function value. 
We use that recursively to find a long path in a fundamental graph, which results in a minor isomorphic to the fan matroid. 

The paper is organized as follows. 
In Section~\ref{sec:prelims}, we will introduce our notations and a few results for matroids, branch-depth, and branch-width.
In Section~\ref{sec:twisted}, we will discuss the concept of twisted matroids introduced by Geelen, Gerards, and Kapoor~\cite{GGK2000}.
In Section~\ref{sec:proof}, we prove our main theorem, Theorem~\ref{thm:main-intro} by finding a `lollipop' inside a twisted matroid.
In Section~\ref{sec:consequences}, we prove its consequences for 
matroids representable over a fixed finite field 
and quasi-graphic matroids.

\section{Preliminaries}\label{sec:prelims}

\subsection{Set systems}

A \emph{set system~$S$} is a pair~${(E, \mathcal{P})}$ consisting  of a finite set~$E$ and a subset~$\mathcal{P}$ of the power set of~$E$. 
We call~$E$ the \emph{ground set of~$S$} and may denote it by~${E(S)}$. 

For ${i \in \{1,2\}}$ let ${S_i = (E_i, \mathcal{P}_i)}$ be set systems. 
A map ${\varphi \colon E_1 \to E_2}$ is an \emph{isomorphism between~$S_1$ and~$S_2$} if it is bijective and~${P \in \mathcal{P}_1}$ if and only if~${\varphi(P) \in \mathcal{P}_2}$. 
We say~$S_1$ and~$S_2$ are \emph{isomorphic} if there is such an isomorphism. 

Given two sets~$X$ and~$Y$, we denote by 
\[
    {X \triangle Y := (X \setminus Y) \cup (Y \setminus X)}
\]
the \emph{symmetric difference of~$X$ and~$Y$}. 

Given a set system~${S = (E, \mathcal{P})}$ and a subset~${X \subseteq E}$ we define 
\[
    \mathcal{P} \triangle X := \{ P \triangle X \colon P \in \mathcal{P} \} 
    \qquad \textnormal{ and } \qquad 
    \mathcal{P} | X := \{ P \subseteq X \colon P \in \mathcal{P} \}.
\]

Given an integer~$n$, we write~${[n]}$ for the set~${\{ i \in \mathbb{Z} \colon 1 \leq i \leq n \}}$ of positive integers up to~$n$.

\subsection{Matroids}

Whitney~\cite{whitney} introduced matroids. 
We mostly follow the notation in~\cite{oxley}. 

A \emph{matroid}~$M$ is a set system~${(E,\mathcal{B})}$ satisfying the following properties: 
\begin{enumerate}
    [label=(B\arabic*)]
    \item\label{axiom:B1} $\mathcal{B}$ is non-empty.
    \item\label{axiom:B2} For every ${B_1, B_2 \in \mathcal{B}}$ and every ${x \in B_1 \setminus B_2}$, there exists an element~$y \in B_2 \setminus B_1$ such that \linebreak ${(B_1 \setminus \{x\}) \cup \{y\} \in \mathcal{B}}$. 
\end{enumerate}

An element of~$\mathcal{B}$ is called a \emph{base} of~$M$. 
We denote the set of bases of a matroid~$M$ by~${\mathcal{B}(M)}$. 
A set~$X$ is \emph{independent} if it is a subset of a base, and we denote the set of independent sets of~$M$ by~${\mathcal{I}(M)}$. 
A set~$X$ is \emph{dependent} if it is not independent. 
A \emph{circuit} is a minimal dependent set, and we denote the set of circuits of~$M$ by~${\mathcal{C}(M)}$. 
The \emph{rank} of a set~$X$ in a matroid~$M$, denoted by~${r_M(X)}$, is defined as the size of a maximum independent subset of~$X$. 
We write~${r(M)}$ to denote~${r_M(E(M))}$, the \emph{rank} of~$M$. 
The rank function satisfies the \emph{submodular inequality}: 
for all~${X, Y \subseteq E(M)}$, 
\begin{equation}\label{eq:submodular}
    r_M(X) + r_M(Y) \geq r_M(X \cap Y) + r_M(X \cup Y).
\end{equation}

The \emph{dual matroid} of~$M$, denoted by~$M^*$, is the matroid on~${E(M)}$ where a set~$B$ is a base of~$M^*$ if and only if~${E(M) \setminus B}$ is a base of~$M$. 
It is well known that 
\[ 
    r_{M^*}(X) = r_M(E(M) \setminus X) + \abs{X} - r(M). 
\]
For a set~${X \subseteq E(M)}$, we write~${M \setminus X}$ for the matroid~${(E(M) \setminus X, \mathcal{B}')}$, 
where~$\mathcal{B}'$ is the set of maximal elements of~${\mathcal{I}(M)|(E(M) \setminus X)}$. 
This operation is called the \emph{deletion}. 
The \emph{contraction} is defined as~${M \slash X = (M^* \setminus X)^*}$. 
The \emph{restriction} is defined as~${M|X = M \setminus (E(M) \setminus X)}$. 
A matroid~$N$ is a \emph{minor} of a matroid~$M$
if~${N = M \setminus X \slash Y}$ for some disjoint subsets~$X$ and~$Y$ of~${E(M)}$. 
Note that if~$N$ is a minor of~$M$, then there are sets~${X, Y \subseteq E(M)}$, where~$X$ is independent in~$M^\ast$ and~$Y$ is independent in~$M$, such that~${N = M \setminus X \slash Y}$, see~\cite{oxley}*{Lemma~3.3.2}. 

The \emph{connectivity function}~$\lambda_M$ of a matroid~$M$ is defined as 
\[
    \lambda_M(X) = r_M(X) + r_M(E(M) \setminus X) - r(M). 
\]
It is easy to check that~${\lambda_{M}(X) = \lambda_{M^*}(X)}$. 

The connectivity function satisfies the following three inequalities. 

\begin{proposition}\label{prop:conn}
    Let~$M$ be a matroid. 
    \begin{enumerate}
        [label=(F\arabic*)]
        \item\label{item:f1} ${0 \leq \lambda_{M}(X) \leq \abs{X}}$ for all~${X \subseteq E(M)}$. 
        \item\label{item:f2} ${\lambda_{M}(X) = \lambda_{M}(E(M) \setminus X)}$ for all~${X \subseteq E(M)}$. 
        \item\label{item:f3} ${\lambda_{M}(X) + \lambda_{M}(Y)\geq \lambda_{M}(X \cap Y) + \lambda_{M}(X\cup Y)}$ for all~${X, Y \subseteq E(M)}$. 
    \end{enumerate}
\end{proposition}

A matroid~$M$ is \emph{connected} if~${\lambda_{M}(X) \neq 0}$ for all non-empty proper subsets~$X$ of~${E(M)}$. 
A \emph{component} of a matroid~$M$ with~${\abs{E(M)} \neq 0}$ is a minimal non-empty set~$X$ such that~${\lambda_M(X) = 0}$, and the empty set is the unique component of the empty matroid~${(\emptyset,\{\emptyset\})}$. 
So a matroid is connected if and only if it has exactly~$1$ component, namely its ground set. 
By a slight abuse of notation, if~$C$ is a component of~$M$, we may also refer to the matroid~${M|C}$ as a component of~$M$.

For a matroid ${M = (E, \mathcal{B})}$, a base~${B \in \mathcal{B}}$, and an element~${e \in E \setminus B}$, 
the \emph{fundamental circuit of~$e$ with respect to~$B$}, denoted by~$C_{M}(e,B)$,
is the circuit that is a subset of~${B \cup \{e\}}$.
It is straightforward to see that such a circuit exists and is unique. 

We omit the subscript~$M$ in~$r_{M}$,~$C_{M}$,~$\lambda_{M}$ if it is clear from the context.

\begin{lemma}
    \label{lem:connfunction}
    If~$N$ is a minor of a matroid~$M$ and~$X$ is a subset of~${E(N)}$, 
    then 
    \[
        \lambda_{M}(X) \leq \lambda_{M}(E(N)) + \lambda_{N}(X).
    \]
\end{lemma}

\begin{proof}
    Let~${N = M \setminus D \slash C}$. 
    We may assume that~$C$ is independent.
    Let~${Z := C \cup D}$. 
    Let~${E := E(M)}$ and let~${Y := E(N) \setminus X}$.
    Let~${r := r_M}$ be the rank function of~$M$. 
    Then 
    \begin{itemize}
        \item ${r_N(X) = r(X \cup C) - r(C)}$, 
        \item ${r_N(Y) = r(Y \cup C) - r(C)}$, and 
        \item ${r_N(X\cup Y) = r(X \cup Y \cup C)-r(C)}$. 
    \end{itemize}
    Since~$r$ is submodular, we deduce the following:
    \begin{align*}
        \lefteqn{\lambda_M(Z) + \lambda_N(X)}\\
        &= r(Z) + r(X \cup Y) - r(E) + r(X \cup C) + r(Y \cup C) - r(X \cup Y \cup C) - r(C)\\
        &= \big( r(X \cup C) + r(X \cup Y) - r(X \cup Y \cup C) \big) + \big( r(Y \cup C) + r(Z) - r(C) \big) - r(E) \\ 
        &\geq r(X) + r(Y \cup Z) - r(E) = \lambda_M(X).\qedhere 
    \end{align*}
\end{proof}

\begin{lemma}
    \label{lem:connfunction2}
    If~$N$ is a minor of a matroid~$M$ and~$X$ is a subset of~${E(M)}$, 
    then 
    \[
        \lambda_{M}(X) \leq \lambda_{N}( X \cap E(N) ) + \abs{ E(M) \setminus E(N) } .
    \]
\end{lemma}

\begin{proof}
    It is enough to prove the inequality when~${\abs{E(M)} - \abs{E(N)} = 1}$.
    We may assume that~${N = M \setminus \{e\}}$ for some~${e \in E(M)}$ by taking the dual if necessary.
    By~\ref{item:f2} applied to~$\lambda_M$ and~$\lambda_{N}$, we may assume that~${e \in X}$.
    Let~$r$ be the rank function of~$M$.
    Observe that ${r(X \setminus \{e\}) + r(\{e\}) \geq r(X)}$ by equation~\eqref{eq:submodular} and~${r(E(M) \setminus \{e\}) \leq r(E(M))}$. 
    Thus 
    \begin{align*}
        \lambda_M(X) 
        &= r(X) + r(E(M) \setminus X) - r(E(M))\\
        &\leq r(X \setminus \{e\}) + r(\{e\}) + r(E(M) \setminus X) - r(E(M) \setminus \{e\})\\
        &\leq \lambda_N(X\setminus\{e\}) + 1. \qedhere
    \end{align*}
\end{proof}

\subsection{Branch-depth}\label{subsec:branchdepth}

We will use the following lemma of DeVos, Kwon, and Oum~\cite{branchdepth}.
Here we state it for matroids.

\begin{lemma}[DeVos, Kwon, and Oum~\cite{branchdepth}*{Lemma~2.3}]
    \label{lem:component}
    Let~${m}$ be a non-negative integer, 
    let~$M$ be a matroid, 
    and let~${\{ E_i \colon i \in [m] \}}$ be a partition of~${E(M)}$ into non-empty sets 
    such that~${\lambda_{M}(E_i) = 0}$ for all~${i \in [m]}$. 
    Let~$k_i$ be the branch-depth of~${M|E_i}$ for~${i \in [m]}$, 
    and let~${k := \max \{ k_i \colon i \in [m] \}}$. 
    Then the branch-depth of~$M$ is~$k$ or~${k + 1}$.
    Moreover, if the branch-depth of~$M$ is~${k+1}$, then it has 
    a~${(k,k+1)}$-decomposition. 
\end{lemma}

\begin{lemma}\label{lem:Xcomponent}
    Let~$m$ be a non-negative integer. 
    Let~$M$ be a 
    matroid of branch-depth~$m$ 
    and let~$X$, $Y$ be disjoint subsets of~${E(M)}$ such that~${X \cup Y \neq \emptyset}$. 
    Then~${M \setminus X \slash Y}$ has a component of branch-depth at least~${m - \abs{X} - \abs{Y}}$. 
\end{lemma}

\begin{proof}
    We follow the idea of \cite{obstructions}*{Lemma 2.6}. 
    If~${\abs{X \cup Y} \geq m}$, then there is nothing to prove. 
    So we may assume that~${0 < \abs{X \cup Y} < m}$. 

    Let~${d := m - \abs{X \cup Y} - 1}$, and suppose that every component of~${M \setminus X \slash Y}$ has 
    branch-depth at most~$d$. 
    Let~${\{ C_i \colon i \in [t] \}}$ be the set of components of~${M \setminus X \slash Y}$. 
    For each~${i \in [t]}$, 
    if~${\abs{C_i} \geq 2}$, then let~${(T_i,\sigma_i)}$ be 
    a~${(d,d)}$-decomposition
    with a node~$r_i$ of~$T_i$ having distance at most~$d$ 
    to every node of~$T_i$. 
    If~${\abs{C_i} = 1}$, then we let~$T_i$ be the one-node graph on~${\{r_i\}}$ and take~${\sigma_i \colon C_i \to \{r_i\}}$. 

    We construct a decomposition~${(T,\sigma)}$ of~$M$ 
    by letting~$T$ be a tree obtained from the disjoint union of all~$T_i$'s 
    by adding a new node~$r$ and adding edges~${rr_i}$ for all~${i \in [t]}$, letting~$\sigma$ map~${v \in C_i}$ to~${\sigma_i(v)}$, and appending~${\abs{X \cup Y}}$ leaves to~$r$ and 
    letting~$\sigma$ map each element of~${X \cup Y}$ to a distinct leaf attached to~$r$. 
    Then~${(T,\sigma)}$ has radius at most~${d + 1}$, which is strictly less than $m$ since $\abs{X\cup Y}\neq 0$. 
    Furthermore,~$r$ has width at most~${\abs{X \cup Y} \leq m - 1}$ by Lemma~\ref{lem:connfunction2}, and all other internal nodes of $T$ have width at most~${d + \abs{X \cup Y} = m - 1}$ 
    by Lemma~\ref{lem:connfunction}. 
    This contradicts our assumption that~$M$ has branch-depth~$m$. 
    Thus we conclude that~${M \setminus X \slash Y}$ has a component inducing a matroid of branch-depth at least~${d+1 = m - \abs{X} - \abs{Y}}$. 
\end{proof}

\begin{lemma}\label{lem:XYcomponents}
    Let~$m$ and~$k$ be non-negative integers, 
    let~$M$ be a matroid, and let~$N_1$ and~$N_2$ be minors of~$M$ such that~${(E(N_1),E(N_2))}$ is a partition of~${E(M)}$ 
    and~${\lambda_{M}(E(N_1)) \leq k}$.
    If all components of both~$N_1$ and~${N_2}$ have branch-depth at most~$m$,
    then~$M$ has branch-depth at most~${\max(m+k,m+2)}$.
\end{lemma}

\begin{proof}
    By Lemma~\ref{lem:component}, both~${N_1}$ and~${N_2}$ have branch-depth at most~${m + 1}$ 
    and if any of them has branch-depth equal to~${m + 1}$, then it has a 
    ${(m,m+1)}$-decomposition.
    
    For~${i \in [2]}$, if~${\abs{E(N_i)} > 1}$, then let~${(T_i,\sigma_i)}$ be a ${(m,m+1)}$-decomposition of~${N_i}$ 
    and let~$r_i$ be a node of~$T_i$ within distance~${m+1}$ from every node of~$T_i$. 
    If~${\abs{E(N_i)} = 1}$, then 
    let~${(T_i,\sigma_i)}$ be the one-node tree on~${\{ r_i \}}$ and take~${\sigma_i \colon E(N_i) \to \{ r_i \}}$. 
    
    Let~$T$ be a tree obtained from the disjoint union of~$T_1$ and~$T_2$
    by adding a new node~$r$ and adding two edges~${rr_1}$ and~${rr_2}$. 
    Let~$\sigma$ be the bijection from~${E(M)} $ to the set of leaves of~$T$ induced by~$\sigma_1$ and~$\sigma_2$. 
    Then~${(T,\sigma)}$ is a decomposition of radius at most~${m + 2}$. 
    Furthermore by Lemma~\ref{lem:connfunction}, the width of~${(T,\sigma)}$
    is at most~${m + k}$. 
    Thus, the branch-depth of~$M$ is at most~${\max(m+k,m+2)}$. 
\end{proof}

\subsection{Branch-width}

Robertson and Seymour~\cite{graphminorsX} introduced the concept of branch-width.
A \emph{subcubic} tree is a tree such that every node has degree~$1$ or~$3$. 
A \emph{branch-decomposition} of a matroid~$M$ is defined as a pair~$ {(T,\sigma)} $  consisting of
a subcubic tree~$T$ and a bijection~$\sigma$ from~${E(M)}$ to the set of leaves of~$T$. 
The \emph{width} of an edge~$e$ in~$T$ is defined as~${\lambda_{M}(A_e)+1}$, 
where~${(A_e,B_e)}$ is the partition of~${E(M)}$ induced by the components of~${T \setminus e}$ under~$\sigma^{-1} $. 
The \emph{width} of a branch-decomposition~${(T,\sigma)} $ is the maximum width of edges in~$T$. 
The \emph{branch-width} of a matroid~$M$, denoted by~$\bw(M)$, is defined to be the minimum integer~$k$ 
for which~$M$ admits a branch-decomposition of width~$k$ if~${E(M)}$ has more than one element, and is defined to be~$1$ otherwise. 

Here is a classical lemma on branch-width. 
For the completeness of this paper, we include its proof. 
An equivalent lemma appears in \cite{ggrw}*{Lemma 4.2}, \cite{os2006}*{Theorem 5.1}.
\begin{lemma}\label{lem:AXYseparation}
    Let~$w$ and~$k$ be positive integers. 
    Let~$M$ be a matroid of branch-width at most~$w$ and let~${Z \subseteq E(M)}$. 
    If~${\abs{Z} \geq 3k + 1}$, 
    then there is a partition~$(X,Y)$ of~$E(M)$ such that
    \[
        \lambda(X) < w
        \quad \text{ and } \quad
        \min ( \abs{Z \cap X}, \abs{Z \cap Y} ) > k. 
    \]
\end{lemma}

\begin{proof}
    Let~${(T,\sigma)} $ be a branch-decomposition of width at most~$w$. 
    We construct a directed multigraph~$D$ on~$V(T)$ as follows. 
    For each edge~${e = uv}$ of~$T$, we add a directed edge from~$w$ to~$v$ if
    the component of~${T - e}$ containing~$v$ has more than~$k$ leaves in~${\sigma(Z)} $. 
    If between two nodes there are edges oriented in both directions, then 
    that gives a desired partition~${(A_e,B_e)}$. 
    So we may assume that ${\abs{E(D)} \leq \abs{E(T)}}$. 
    Since~${\abs{E(T)} < \abs{V(T)}}$, there is a node~$v$ of~$T$ having no outgoing edges in~$D$.  
    Since~${k \geq 1}$, every edge of~$D$ incident with a leaf of~$T$ is oriented away from the leaf and therefore~$v$ is an internal node of~$T$.  
    However~$v$ has degree~$3$ in~$T$, and so~${\abs{Z} \leq 3k}$, 
    contrary to the assumption that~${\abs{Z} > 3k}$. 
\end{proof}

The following lemma is well{-}known and is an easy consequence of the definitions.

\begin{lemma}[Dharmatilake~\cite{Dharmatilake1996}]
    \label{lem:branch-width-minor-closed}
    If~$N$ is a minor of~$M$, then 
    the branch-width of~$N$ is at most the branch-width of~$M$. 
    \qed
\end{lemma}

The branch-width of a graph~$G$ is defined as follows.
Let~$T$ be a subcubic tree, and let~$\sigma $ be a bijection from~${E(G)}$ to the set of leaves of~$T$. 
Then we say that~${(T,\sigma)} $ is a \emph{branch-decomposition} of~$G$. 
Let~$e$ be an edge of~$T$, and let~${(A_{e},B_{e})}$ be a partition of~${E(G)}$ induced by the components of~${T \setminus e}$.
The \emph{width} of~$e$ is the number of vertices that are incident with edges in both~$A_{e}$ and~$B_{e}$.
The \emph{width} of the branch-decomposition is the maximum width of an edge in~$T$. 
The \emph{branch-width} of~$G$ is the minimum integer~$k$ such that~$G$ has a branch-decomposition of width~$k$ when~$G$ has at least two edges (otherwise the branch-width of~$G$ is defined to be~$0$). 

Hicks and McMurray~\cite{HM07} and, independently, Mazoit and Thomass\'{e}~\cite{MT07} proved that the branch-width of the graph~$G$ is equal to the branch-width of the graphic matroid~${M(G)}$, if~$G$ has a cycle of length at least~$2$.

\section{Fundamental graphs and twisted matroids}\label{sec:twisted}

\subsection{The fundamental graph}

Let~$M$ be a matroid on ground set~$E$ and let~$B$ be a base of~$M$. 
We define the \emph{fundamental graph~${G(M,B)}$ of~$M$ with respect to~$B$} 
as the bipartite graph with bipartition classes~$B$ and~${E \setminus B}$ such that there is an edge between~${b \in B}$ and~${e \in E\setminus B}$ if and only if~$b$ is in the fundamental circuit~${C_M(e,B)}$ of~$e$ with respect to~$B$. 

The following statements about the fundamental graph are well known and are easy consequences of the relevant definitions. 

\begin{proposition}
    \label{prop:fundamental-graph}
    Let~$M$ be a matroid and let~$B$ be a base of~$M$. 
    Then the following statements are true. 
    \begin{enumerate}
        [label=(\roman*)]
        \item \label{rmk:connectivity} 
            $M$ is connected if and only if~${G(M,B)}$ is connected 
            (see \cite{oxley}*{Proposition 4.3.2}). 
        \item \label{rmk:dual-fundamental} 
            ${G(M,B)}$ and ${G(M^*, E(M) \setminus B)}$ are equal. 
    \end{enumerate}
\end{proposition}

It is well known that a matroid is binary if and only if for any base~$B$, any circuit~$C$ is the symmetric difference of all fundamental circuits~${C(e,B)}$ with~${e \in C \setminus B}$ \cite{oxley}*{Theorem~9.1.2}. 
Hence, every binary matroid is completely determined by its fundamental graph and a colour class of any proper $2$-colouring of that fundamental graph that is a base of the matroid. 

For general matroids, such a complete determination fails; 
two distinct matroids may have the same fundamental graph with respect to the same base.
But one can ask how a fundamental graph with respect to some base will change when doing base exchange. 

Note that if~${G(M,B)}$ has an edge~$uv$, then ${B' := B \triangle \{u,v\}}$ is a base of~$M$. 
The operation of constructing ${G(M, B')}$ from~${G(M,B)}$ is called a \emph{pivot} on~$uv$. 

\begin{proposition}[Geelen, Gerards, and Kapoor~\cite{GGK2000}]
    \label{prop:fundamental-graph-rules1}
    Let~$M$ be a matroid, let~$B$ be a base of~$M$, and let~$uv$ be an edge of~${G := G(M,B)}$. 
    Then with~${B' := B \triangle \{u,v\}}$ the following statements about~${G' := G(M, B')}$ are true.
    \begin{enumerate}
        [label=(\roman*)]
        \item 
            $\n_{G'}(u) = \n_{G}(v)\triangle \{u,v\}$, 
            and $\n_{G'}(v) = \n_{G}(u)\triangle \{u,v\}$.
        \item 
            If ${x \notin \n_G(u) \cup \n_G(v)}$, then $\n_{G'}(x) = \n_{G}(x)$. 
        \item 
            If ${x \in \n_G(u)}$ and ${y \in \n_G(v) \setminus \n_G(x)}$, then~$xy$ is an edge of~$G'$. 
            
        \item 
            If ${G[\{x,y,u,v\}]}$ is a cycle of length~$4$, then~$xy$ is an edge of~$G'$ if and only if~${B \triangle \{x,y,u,v\}}$ is a base of~$M$. 
    \end{enumerate}
\end{proposition}

Note that the first three rules of the proposition allow us to determine from the graph~$G$ the adjacency or non-adjacency in~$G'$ of some pairs of vertices~$x$ and~$y$. 
This is not true of the fourth rule. 
However, if in addition to the edge set of the fundamental graph we were given a list of `hyperedges' ${\{x,y,u,v\}}$ for which ${B \triangle \{x,y,u,v\}}$ is a base, then we could apply all four rules. 

As an extension of that idea, Geelen, Gerards, and Kapoor~\cite{GGK2000} introduced \emph{twisted matroids}, 
which can in a sense be viewed as `fundamental hypergraphs'. 
We introduce their machinery in the next subsection. 

\subsection{Twisted matroids}

Let~${S = (E,\mathcal{P})}$ be a set system and let~${X \subseteq E}$. 
We define the \emph{twist of~$S$ by~$X$} as 
\[
    S \ast X := ( E , \mathcal{P} \triangle X ).
\]
Moreover, we define the \emph{restriction of~$S$ to~$X$} as
\[
    S[X] := (E, \mathcal{P} | X),
\]
where, as noted before, ${\mathcal{P} | X := \{ P \subseteq X \colon P \in \mathcal{P} \}}$. 

\begin{remark}
    \label{rmk:double-twist}
    Let ${S = (E, \mathcal{P})}$ be a set system and let~${X, Y \subseteq E}$. 
    Then 
    \[
        (S \ast X) \ast Y = S \ast (X \triangle Y).
    \]
\end{remark}

\vspace{0.2cm}

\pagebreak[2]

A \emph{twisted matroid}~$W$ is a set system~$(E,\mathcal{F})$ satisfying the following properties: 
\begin{enumerate}
    [label=\upshape{(T\arabic*)}]
    \item\label{axiom:T1} ${\emptyset \in \mathcal{F}}$. 
    \item\label{axiom:T2} For every~${F_1, F_2 \in \mathcal{F}}$ and every~${e \in F_1 \triangle F_2}$, there is an~${f \in F_1 \triangle F_2}$ such that~${F_1 \triangle \{e,f\} \in \mathcal{F}}$.
    \item\label{axiom:T3} There is a set~${B \subseteq E}$ such that~${\abs{B \cap F} = \abs{ (E \setminus B) \cap F}}$ for all~${F \in \mathcal{F}}$. 
\end{enumerate}

We call~$E$ the \emph{ground set} of~$W$ and may denote it by~${E(W)}$. 
We call the elements of~$\mathcal{F}$ \emph{feasible} (with respect to~$W$), and may denote the set~$\mathcal{F}$ by~${\mathcal{F}(W)}$. 
We call a set~$B$ which satisfies~\ref{axiom:T3} a \emph{base of~$W$}. 
We denote by~${\mathcal{B}(W)}$ the set of bases of~$W$. 

We observe that~\ref{axiom:T3} implies that every feasible set has even size. 
And in fact it is enough to restrict our attention to feasible sets of size two, as the following proposition will show. 

\begin{proposition}
    \label{prop:weaker-axiom}
    Let~${W = (E,\mathcal{F})}$ be a set system satisfying~\ref{axiom:T1} and~\ref{axiom:T2}. 
    Then~\ref{axiom:T3} is equivalent to the following axiom. 
    \begin{enumerate}
        [label=\upshape{(T\arabic*$'$)}] 
        \setcounter{enumi}{2}
        \item \label{axiom:T3'} 
            \textnormal{There is a set~${B \subseteq E}$ such that~${\abs{B \cap F} = \abs{ (E \setminus B) \cap F}}$ for all~${F \in \mathcal{F}}$ with $\abs{F} \leq 2$.}
    \end{enumerate}
    
    Moreover, a set~$B$ satisfies~\ref{axiom:T3} if and only if it satisfies~\ref{axiom:T3'}. 
\end{proposition}

\begin{proof}
    Assume~\ref{axiom:T3'} holds and let~${B \subseteq E}$ be as required. 
    We will show that~$B$ satisfies~\ref{axiom:T3}. 
    Suppose for a contradiction 
    that there is a set~${F \in \mathcal{F}}$ of minimum size 
    for which~${\abs{B \cap F} \neq \abs{(E \setminus B) \cap F}}$. 
    Let~${X, Y \in \{B , (E \setminus B)\}}$ with~${\abs{X \cap F} < \abs{Y \cap F}}$. 
    By applying~\ref{axiom:T2} to~$\emptyset$,~$F$, and some~${e \in F}$, there is an~${f \in F}$ such that~${\{e,f\} = \emptyset \triangle \{e,f\} \in \mathcal{F}}$ and hence~${e \neq f}$ by~\ref{axiom:T3'}. 
    By~\ref{axiom:T3'}, exactly one of~$e$ or~$f$ is in~$B$, so~${X \cap F}$ is non-empty. 
    Applying~\ref{axiom:T2} again to~$F$,~$\emptyset$, and some~${x \in X \cap F}$, there is some~${z \in F}$ such that 
    \[
        {F' := F \triangle \{x,z\} = F \setminus \{x,z\} \in \mathcal{F}}.
    \]
    Now 
    \[
        \abs{F' \cap X} 
        < \abs{F \cap X} 
        \leq \abs{F \cap Y} - 1 
        \leq \abs{F' \cap Y},
    \]
    contradicting that~$F$ was a set of minimum size for which~${\abs{B \cap F} \neq \abs{(E \setminus B) \cap F}}$. 
\end{proof}

Note that this axiomatic definition of twisted matroids does not coincide with the original definition of Geelen, Gerards, and Kapoor~\cite{GGK2000}, in which they defined twisted matroids to be the twist~${M \ast B}$ of a matroid~$M$ with a base~$B$ of~$M$. 
The following proposition establishes together with Remark~\ref{rmk:double-twist} the equivalence of these definitions. 

\begin{proposition}
    \label{prop:twisted-equivalence}
    Let~${M = (E,\mathcal{B})}$ and~${W = (E,\mathcal{F})}$ be set systems and let~${B \subseteq E}$ such that~${W = M \ast B}$ (equivalently~${M = W \ast B}$). 
    Then the following statements are equivalent. 
    \begin{enumerate}
        [label=(\alph*)]
        \item \label{prop-item:equi-matroid} 
            $M$ is a matroid and~${B \in \mathcal{B}}$. 
        \item \label{prop-item:equi-twisted} 
            $W$ is a twisted matroid and ${B \in \mathcal{B}(W)}$. 
    \end{enumerate}
\end{proposition}

\begin{proof}
    Suppose~\ref{prop-item:equi-matroid} holds. 
    Then~${\emptyset \in \mathcal{B} \triangle B = \mathcal{F}}$. 
    Since~$M$ satisfies~\ref{axiom:B2},~$W$ satisfies~\ref{axiom:T2}. 
    Finally, every element of~${\mathcal{F} \triangle B = \mathcal{B}}$ has size~${\abs{\emptyset \triangle B} = \abs{B}}$, so~$B$ satisfies~\ref{axiom:T3}. 
    
    Suppose~\ref{prop-item:equi-twisted} holds.
    By~\ref{axiom:T1} 
    we have~${B \in \mathcal{B}}$ and hence~\ref{axiom:B1} holds. 
    For~\ref{axiom:B2}, consider bases~${B_1, B_2 \in \mathcal{B}}$ and~${e \in B_1 \setminus B_2}$.
    Let~${F_1 := B_1 \triangle B}$ and~${F_2 := B_2 \triangle B}$. 
    Then~${F_1, F_2 \in \mathcal{F}}$. 
    Note that 
    \[
        {(B_1 \setminus B_2) \cup (B_2 \setminus B_1) = F_1 \triangle F_2}.
    \]
    Then~${e \in F_1 \triangle F_2}$ and hence by~\ref{axiom:T2} there is an~${f \in F_1 \triangle F_2}$ such that~${F_1 \triangle \{e,f\} \in \mathcal{F}}$. 
    If~${f \in B_1 \setminus B_2}$, then either~$F_1$ or~${F_1 \triangle \{e,f\}}$ will contradict~\ref{axiom:T3}. 
    Hence, ${f \in B_2 \setminus B_1}$. 
\end{proof}

For a twisted matroid~$W$ we define 
\[
    \mathcal{M}(W) := \big\{ (M,B) \colon \textnormal{$M$ is a matroid and~$B$ is a base of~$M$ such that } M \ast B = W \big\}.
\]
If ${(M,B) \in \mathcal{M}(W)}$, then we say~$M$ is \emph{associated with~$W$}. 
Note that for a twisted matroid~$W$ and a base~$B$ of~$W$ we have~${(W \ast B, B) \in \mathcal{M}(W)}$ by Proposition~\ref{prop:twisted-equivalence} and Remark~\ref{rmk:double-twist}. 

We define the \emph{fundamental graph~$G(W)$ of~$W$} as the graph on vertex set~$E$ such that there is an edge between~${x, y \in E}$ if and only if ${\{x,y\} \in \mathcal{F}(W)}$. 
By~\ref{axiom:T3}, the fundamental graph of~$W$ is bipartite. 
In fact, by~\ref{axiom:T3'}, a set~$B$ is a base of~$W$ if and only if there is a proper~$2$-colouring of~$G(W)$ in which~$B$ is a colour class. 

\begin{proposition}
    \label{prop:twisted1}
    Let~${W = (E, \mathcal{F})}$ be a twisted matroid, 
    let~${(M,B) \in \mathcal{M}(W)}$ 
    and~${X \subseteq E}$. 
    Then the following statements are true.
    \begin{enumerate}
        [label=(\roman*)]
        \item \label{prop-item:feasible} 
            ${X \in \mathcal{F}(W)}$ if and only if ${B \triangle X \in \mathcal{B}(M)}$. 
        \item \label{prop-item:fundamental-graph} 
            The fundamental graph~${G(M,B)}$ is equal to 
            the fundamental graph~${G(W)}$. 
        \item \label{prop-item:accociated-dual} 
            ${(M^*, E \setminus B) \in \mathcal{M}(W)}$. 
        \item \label{prop-item:associated-connected} 
            If $M$ is connected, then $\mathcal{M}(W)=\{(M,B),(M^*, E \setminus B)\}$.
    \end{enumerate}
\end{proposition}

\begin{proof}
    For~\ref{prop-item:feasible}, 
    suppose~${X \in \mathcal{F}}$. 
    Then~$M$ has a base~${B'}$ such that~${X = B \triangle B'}$. 
    By the properties of the symmetric difference we obtain~${B' = B \triangle X}$. 
    Conversely, if ${B \triangle X = B' \in \mathcal{B}}$, then~${X = B' \triangle B \in \mathcal{F}}$. 
    
    For~\ref{prop-item:fundamental-graph}, 
    note that~$f$ is on the fundamental circuit of~$e$ with respect to~$B$, if and only if~${(B \setminus \{f\}) \cup \{e\}}$ is a base of~$M$. 
    So ${\{e,f\} = ((B \setminus \{f\}) \cup \{e\}) \triangle B}$ is feasible if and only if~${ef}$ is an edge of~${G(M,B)}$. 
    
    For~\ref{prop-item:accociated-dual}, 
    note that for~${(M, B) \in \mathcal{M}(W)}$ we have 
    \[
        {\mathcal{B}(M^*) 
        = \mathcal{B}(M) \triangle E(M)
        = (\mathcal{F} \triangle B) \triangle E(M) 
        = \mathcal{F} \triangle (E(M) \setminus B)}.
    \]
    
    For~\ref{prop-item:associated-connected}, suppose~$M$ is connected. 
    By Proposition~\ref{prop:fundamental-graph}\ref{rmk:dual-fundamental}, 
    ${G(M,B) = G(W)}$ is connected and hence every proper 2-colouring of~${G(W)}$ has~$B$ and~${E \setminus B}$ as its colour classes. 
\end{proof}

\subsection{Minors of twisted matroids}

\begin{proposition}
    \label{prop:twisted-minors} 
    Let~${W = (E, \mathcal{F})}$ be a twisted matroid and let~${X, F \subseteq E}$. 
    Then the following statements are true. 
    \begin{enumerate}
        [label=(\roman*)]
        \item \label{prop-item:feasible-twist} 
            ${W \ast F}$ is a twisted matroid if and only if~${F \in \mathcal{F}}$. \\
            Additionally, 
            ${W \ast F = M \ast (B \triangle F)}$ for any~${(M,B) \in \mathcal{M}(W)}$. 
        \item \label{prop-item:restriction} 
            ${W[X]}$ is a twisted matroid for which  
            \[
                {G(W[X]) = G(W)[X]} 
                \ \text{ and } \ 
                {\{ B \cap X \colon B \in \mathcal{B}(W) \} \subseteq \mathcal{B}(W[X])}
            \]
        \item \label{prop-item:restricted-twist} 
            If ${F \in \mathcal{F}|X}$, then~${W[X] \ast F = (W \ast F)[X]}$. 
        \item \label{prop-item:compatible-minor} 
            For~${(M,B) \in \mathcal{M}(W)}$, we have
            \[
                {W[X] \ast (B \cap X) = \big( M \slash (B \setminus X) \big) | X}.
            \]
    \end{enumerate}
\end{proposition}

\begin{proof}
    For~\ref{prop-item:feasible-twist}, 
    suppose~${F \in \mathcal{F}}$. 
    If~${(M,B) \in \mathcal{M}(W)}$, then 
    \[
        {W \ast F = (M \ast B) \ast F = M \ast (B \triangle F)}.
    \]
    Now~${W \ast F}$ is a twisted matroid by Propositions~\ref{prop:twisted1}\ref{prop-item:feasible} and~\ref{prop:twisted-equivalence}. 
    Conversely, suppose that~${W \ast F}$ is a twisted matroid. 
    Now~${\emptyset \in \mathcal{F}(W \ast F) = \mathcal{F} \triangle F}$ by~\ref{axiom:T1}. 
    Hence~${F \in \mathcal{F}}$. 
    
    Both~\ref{prop-item:restriction} and~\ref{prop-item:restricted-twist} are trivial consequences of the definitions. 
    
    For~\ref{prop-item:compatible-minor}, 
    note that if~$B'$ is a base of~$M$ for which~${F'  := B' \triangle B \subseteq X}$, 
    then  
    \[
        {B' \cap X 
        = ((B' \triangle B) \triangle B) \cap X 
        = ((B' \triangle B) \cap X) \triangle (B \cap X) 
        = F'  \triangle (B \cap X)}, 
    \]
    and hence 
    \[
        \big\{ B' \cap X \colon B' \in \mathcal{B}(M) \textnormal{ and } B' \triangle B \subseteq X \big\} 
        = 
        \big\{ F'  \triangle (B \cap X) \colon F'  \in \mathcal{F}|X \big\}. 
    \] 
    Now since
    \[
        {\mathcal{B}\big((M \slash (B \setminus X))|X\big) = \big\{ B' \cap X \colon B' \in \mathcal{B}(M) \textnormal{ and } B' \triangle B \subseteq X \big\}}
    \]
    and
    \[
         \mathcal{F}(W[X]) \triangle (B \cap X)  
        = \big\{ F'  \triangle (B \cap X) \colon F'  \in \mathcal{F}|X \big\}, 
    \]
    we obtain~\ref{prop-item:compatible-minor}. 
\end{proof}

A twisted matroid~$U$ is called a \emph{minor} of a twisted matroid~${W = (E, \mathcal{F})}$ if there are sets~${X \subseteq E}$ and~${F \in \mathcal{F}}$ such that~${U = (W \ast F)[X]}$. 

\begin{proposition}
    \label{prop:twisted-minor-associated}
    Let~$U$ and~$W$ be twisted matroids. 
    Then the following statements are true. 
    \begin{enumerate}
        [label=(\roman*)]
        \item \label{prop-item:twisted-minor-associated1} 
            If~$U$ is a minor of~$W$, then every matroid~$M$ associated with~$W$ has a minor~$N$ associated with~$U$. 
        \item \label{prop-item:twisted-minor-associated2} 
            If some matroid~$M$ associated with~$W$ has a minor~$N$ associated with~$U$, then~$U$ is a minor of~$W$. \\ 
            In particular,~$M$ has a base~$B$ such that~${(M \ast B)[E(U)] = N \ast (B \cap E(U)) = U}$.
    \end{enumerate}
\end{proposition}

\begin{proof}
    For~\ref{prop-item:twisted-minor-associated1}, let~${(M, B') \in \mathcal{M}(W)}$ and~${X \subseteq E}$ and~${F \in \mathcal{F}}$ such that~${U = (W \ast F)[X]}$. 
    Then  
    \[
        {M = W \ast B' 
        = W \ast (B' \triangle F \triangle F)
        = (W \ast F) \ast (B' \triangle F)}, 
    \]
    and hence for~${B := B' \triangle F}$ we have~${(M, B) \in \mathcal{M}(W \ast F)}$.
    By Proposition~\ref{prop:twisted-minors}\ref{prop-item:compatible-minor}, we have that~${U \ast (B \cap X)}$ is a minor of~$M$, as desired. 
    
    For~\ref{prop-item:twisted-minor-associated2}, suppose for ${(M,B') \in \mathcal{M}(W)}$ and~${(N,B'') \in \mathcal{M}(U)}$ we have that~$N$ is a minor of~$M$. 
    Note that by the Scum Theorem~\cite{oxley}*{Theorem~3.3.1} there is a set~${Z \subseteq E(M) \setminus E(N)}$ such that~${N = (M \slash Z) | E(N)}$ and the rank of~${M \slash Z}$ is equal to the rank of~$N$. 
    Without loss of generality~$Z$ is independent. 
    Now~$B''$ is independent in~${M \slash Z}$.
    By the equality of the ranks, ${B := B'' \cup Z}$ is a base of~$M$, 
    and so~${N = (M \slash (B \setminus E(N)))|E(N)}$. 
    Now since~${F := B \triangle B' \in \mathcal{F}(W)}$ we obtain that~${(M,B) \in \mathcal{M}(W \ast F)}$, 
    and hence by Proposition~\ref{prop:twisted-minors}\ref{prop-item:compatible-minor} 
    \[
        {(W \ast F)[E(N)] \ast (B \cap E(N)) = N = U \ast B'' = U \ast (B \cap E(N))}. 
    \]
    Therefore,~${(M \ast B)[E(U)] = (W \ast F)[E(N)] = U}$, as desired. 
\end{proof}

Lastly, let us remark that the minor relation of twisted matroids is transitive. 

\begin{proposition}
    \label{prop:twist-transitivity}
    Let~${W = (E, \mathcal{F})}$ be a twisted matroid, 
    let~${X' \subseteq X \subseteq E}$, 
    let~${F \in \mathcal{F}}$, 
    and~${F' \in \mathcal{F}(W \ast F)|X}$. 
    Then 
    \[
        {F \triangle F' \in \mathcal{F}} \ \text{ and } \ {(W \ast (F \triangle F'))[X'] = (((W \ast F)[X]) \ast F')[X']}. 
    \]
\end{proposition}

\begin{proof}
    By Proposition~\ref{prop:twisted-minors}\ref{prop-item:feasible-twist}, ${(W \ast F)\ast F'}$ is a twisted matroid. 
    Since ${(W \ast F) \ast F' = W \ast (F \triangle F')}$, we have that~${F \triangle F' \in \mathcal{F}}$ again by Proposition~\ref{prop:twisted-minors}\ref{prop-item:feasible-twist}.
    
    Now by Proposition~\ref{prop:twisted-minors}\ref{prop-item:restricted-twist} we have
    \[
        (((W \ast F)[X]) \ast F')[X'] 
        = (((W \ast F) \ast F')[X])[X'] 
        = (W \ast (F \triangle F'))[X']. 
        \qedhere
    \]
\end{proof}

\subsection{More on the fundamental graph and twisted matroids}

\begin{proposition}\label{prop:brualdi-krogdahl}
    Let~${W = (E, \mathcal{F})}$ be a twisted matroid and let~${X \subseteq E}$. 
    Then the following statements are true. 
    \begin{enumerate}
        [label=(\roman*)]
        \item\label{prop-item:brualdi} 
            If~$X \in \mathcal{F}$, then~${G(W)[X]}$ has a perfect matching 
            \textnormal{(Brualdi~\cite{brualdi})}. 
        \item\label{prop-item:krogdahl}
            If~${G(W)[X]}$ has a unique perfect matching, then~${X \in \mathcal{F}}$ 
            \textnormal{(Krogdahl~\cite{krogdahl})}. 
    \end{enumerate}
\end{proposition}

We deduce the following two propositions easily from the above proposition.

\begin{proposition}\label{prop:uniqpm}
    Let~$M_1$ and~$M_2$ be matroids on the common ground set~$E$ sharing a common base~$B$. 
    If the fundamental graphs of~$M_1$ and~$M_2$ with respect to~$B$ are equal and have no cycles, 
    then~${M_1 = M_2}$. 
\end{proposition}

\begin{proof}
    Let~${X \subseteq E}$. 
    Let~$G$ denote the common fundamental graph with respect to~$B$, 
    which by Proposition~\ref{prop:twisted1}\ref{prop-item:fundamental-graph} is equal to~$G(M_1 \ast B) = G(M_2 \ast B)$. 
    Since~$G$ is a forest, every induced subgraph 
    has at most one perfect matching and so for~${i \in [2]}$
    by Proposition~\ref{prop:brualdi-krogdahl}, 
    $X$ is feasible in~${M_i \ast B}$ if and only if~${G[X]}$ has a perfect matching.
    Therefore,~${M_1 \ast B = M_2 \ast B}$, and hence~${M_1 = M_2}$.
\end{proof}

\begin{proposition}\label{prop:path-to-fan}
    Let~$n$ be a positive integer. 
    A matroid~$M$ has a minor isomorphic to~${M(F_n)}$ if and only if~$M$ has a base~$B$ such that~$G(M,B)$ has an induced path on~${2n-1}$ vertices, 
    starting and ending in~$B$. 
\end{proposition}

\begin{proof}
    Suppose that~$M$ has a minor~$N$ isomorphic to~${M(F_n)}$.
    Let~$B'$ be the base of~$N$ corresponding to the star~$K_{1,n}$ in~${M(F_n)}$.
    Then the fundamental graph of~$N$ with respect to~$B'$ is a path
    on~${2n-1}$ vertices, starting and ending in~$B'$.
    By Proposition~\ref{prop:twisted-minor-associated}\ref{prop-item:twisted-minor-associated2},~$M$ has a base~$B$ such that~${(M \ast B)[E(N)] = N \ast (B \cap E(N))}$. 
    Now~${B \cap E(N)}$ is a base of~$N$ by Proposition~\ref{prop:twisted-equivalence}, so ${F := (B \cap E(N)) \triangle B'}$ is in~${\mathcal{F}(M \ast B)|E(N)}$, and hence in~${\mathcal{F}(M\ast B)}$. 
    It follows that~${\hat{B} := B \triangle F}$ is a base of~$M$ by Proposition~\ref{prop:twisted1}\ref{prop-item:feasible}. 
    Note that 
    \[
       \hat{B} \cap E(N) = (B \triangle ((B \cap E(N)) \triangle B')) \cap E(N) = B', 
    \]  
    and hence~${(M \ast \hat{B})[E(N)] = N \ast B'}$. 
    Thus~${G(M,\hat{B})[E(N)] = G(N,B')}$ by Propositions~\ref{prop:twisted1}\ref{prop-item:fundamental-graph} and~\ref{prop:twisted-minors}\ref{prop-item:restriction}, which is the required path. 
    
    Conversely, suppose there is a base~$B$ of~$M$ and a set~${X \subseteq E(M)}$ such that~${G(M,B)[X]}$ is a path on~${2n - 1}$ vertices,
    starting and ending in~$B$. 
    Let~${W := M \ast B}$ and~${U := (M \ast B)[X]}$. 
    Since~$U$ is a minor of~$W$, for some~${B' \in \mathcal{B}(U)}$ the matroid~${N := U \ast B'}$ is a minor of~$M$ by Proposition~\ref{prop:twisted-minor-associated}\ref{prop-item:twisted-minor-associated1}. 
    Again, by Propositions~\ref{prop:twisted1}\ref{prop-item:fundamental-graph} and~\ref{prop:twisted-minors}\ref{prop-item:restriction}, we have that 
    \[
        {G(M,B)[X] = G(U) = G(N,B')}.
    \]
    Since~${G(M,B)[X]}$ is a forest, by Proposition~\ref{prop:uniqpm}, 
    there is a unique matroid with base~$B'$ whose fundamental graph is~${G(M,B)[X]}$,
    and that matroid is isomorphic to~$M(F_n)$, as desired. 
\end{proof}

\vspace{0.2cm}

We will need the following result about the change of the fundamental graph of a twisted matroid when twisting with a feasible set, which in particular will not change for the vertices not involved in the twist. 

\begin{proposition}\label{prop:fundamental-graph-outside-twist}
    Let~${W = (E,\mathcal{F})}$ be a twisted matroid, 
    let~${F \in \mathcal{F}}$, 
    let~${e \in E \setminus F}$ such that there is no~${f \in F}$ for which~${\{e,f\}}$ is feasible, 
    and let~${x \in E}$. 
    Then~${\{e,x\} \in \mathcal{F}}$ if and only 
    if~${\{e,x\} \in \mathcal{F} \triangle F}$. 
\end{proposition}

\begin{proof}
    If~${\{e,x\} \in \mathcal{F} \triangle F}$, then~${F' := \{e,x\} \triangle F \in \mathcal{F}}$.
    Since ${e \in F'}$, by~\ref{axiom:T1} and~\ref{axiom:T2} there is a~${y \in F'}$ such that~${\{e,y\} \in \mathcal{F}}$. 
    Now by the premise of this proposition, ${y = x}$, as desired. 
    
    If~${\{e,x\} \in \mathcal{F}}$, then applying~\ref{axiom:T1} and~\ref{axiom:T2} for ${W \ast F}$, 
    there is a ${y \in \{e,x\} \triangle F}$ for which~${\{e,y\} \in \mathcal{F} \triangle F}$. 
    So by the previous paragraph, ${\{e,y\} \in \mathcal{F}}$ and hence~${y \notin F}$. 
    But then ${y = x}$, as desired. 
\end{proof}

\vspace{0.2cm}

If a matroid property is invariant under
dualising components 
then for a twisted matroid~$W$ that property will be shared by every matroid associated with~$W$. 
So for such properties we are justified to call these properties of the twisted matroid. 

For example, we have the following proposition. 

\begin{proposition}\label{prop:twisted-matroid-properties}
    Let~${W}$ be a twisted matroid. 
    Then for all~${(M,B), (M',B') \in \mathcal{M}(W)}$ and all~${X \subseteq E(W)}$ the following statements are true. 
    \begin{enumerate}
        [label=(\roman*)]
        \item $\lambda_{M'}(X) = \lambda_{M}(X)$. 
        \item $X$ is a component of~$M'$ if and only if~$X$ is a component of~$M$. 
        \item $M'$ is connected if and only if~$M$ is connected. 
        \item The branch-depth of~$M'$ is equal to the branch-depth of~$M$. 
        \item The branch-width of~$M'$ is equal to the branch-width of~$M$. 
        \qed
    \end{enumerate}
\end{proposition}

Motivated by this proposition, we make the following definitions for a twisted matroid~${W}$. 
Let ${(M,B) \in \mathcal{M}(W)}$ be arbitrary. 
We define the \emph{connectivity function}~$\lambda_{W}$ of~$W$ as~$\lambda_{M}$. 
A \emph{component} of~$W$ is a component of~$M$, and~$W$ is \emph{connected} if~$M$ is connected. 
We define the \emph{branch-depth} and the \emph{branch-width} of~$W$, respectively, as the branch-depth and the branch-width of~$M$, respectively. 

Given these definitions, the related results for matroids in Section~\ref{sec:prelims} also hold for twisted matroids, and we will apply them for twisted matroids without further explanation.

\section{Lollipop minors of twisted matroids}\label{sec:proof}

In this section we complete the proof of Theorem~\ref{thm:main-intro}. 
To do this, we introduce the following class of twisted matroids. 

\subsection{Lollipops}\label{subsec:lollipops}

\begin{definition}\label{def:lollipop}
    Let~$a$,~$b$ be non-negative integers. 
    A twisted matroid~$L$ on ground set~${S \dot\cup \{z\} \dot\cup C}$ 
    with fundamental graph ${G := G(L)}$
    is called an \emph{${(a,b)}$-lollipop} 
    if 
    \begin{enumerate}
        [label=(\arabic*)]
        \item ${\abs{S} \geq a}$; 
        \item $G$ is connected;
        \item ${G[S \cup \{z\}]}$ is a path with terminal vertex~$z$; 
        \item ${G[C]}$ is a connected component of~${G-z}$; and
        \item ${L[C]}$ has branch-depth at least~$b$. 
    \end{enumerate}
    We call the tuple~${(S,z,C)}$ the \emph{witness} of~$L$, 
    the twisted matroid~${L[C]}$ the \emph{candy} of~$L$, 
    and the graph~${G[S \cup \{z\}]}$ the \emph{stick} of~$L$. 
\end{definition}

In order to prove Theorem~\ref{thm:main-intro}, we prove the following theorem.

\begin{restatable}{theorem}{lollipops}\label{thm:lollipops}
    For all non-negative integers~$a$, $b$, and~$w$, 
    there is an integer~$d$ 
    such that every twisted matroid~$W$ 
    of branch-width at most~$w$ and branch-depth at least~$d$ 
    has a minor which is an~${(a,b)}$-lollipop. 
    Moreover, we may take~${d = b + 3(3w)^a}$.
\end{restatable}

As we noted in Proposition~\ref{prop:path-to-fan}, induced paths in the fundamental graph are the correct object to look for when looking for fan matroids as a minor of a matroid. 
So lollipops are defined in terms of a long path in the fundamental graph to recover these minors, as we note in the following corollary of Proposition~\ref{prop:path-to-fan}. 
This corollary also shows that
Theorem~\ref{thm:lollipops} implies Theorem~\ref{thm:main-intro}.

\begin{corollary}\label{cor:lollipop-fan}
    Let~$n$ be a positive integer and let~$L$ be a~${(2n,0)}$-lollipop. 
    Then every matroid associated with~$L$ contains a minor isomorphic to~${M(F_n)}$. \qed
\end{corollary}

We also remark that Theorem~\ref{thm:main-intro} implies Theorem~\ref{thm:lollipops} because for all non-negative integers~$a$
and~$b$, there is an integer~$n$ such that for some base~$B$ of~${M(F_n)}$ the twisted matroid~${M(F_n) \ast B}$ is an ${(a,b)}$-lollipop, since~$M(F_n)$ has large branch-depth. 

The reason for considering lollipops as opposed to fan matroids is that it allows an inductive approach to find lollipop minors in twisted matroids of sufficiently high branch-depth. 
If we find a lollipop whose candy has sufficiently high branch-depth, then we can iteratively find another lollipop as a minor of the candy. 
Having found a large enough number of such `nested' lollipops, the small branch-width allows us to identify the stick of one of the lollipops which we are able to `extend' into its candy and `attach' it to a deletion minor of sufficiently large branch-depth. 
This way, we are able to inductively find a lollipop with a longer stick. 

Since lollipops are defined as twisted matroids, the choice of a base of the original matroid is important. 
However, the following corollary of the results of the previous section allows us a large amount of flexibility in exchanging parts of the base of the matroid associated with the candy. 

\begin{corollary}\label{cor:lollipop-twist}
    Let~$a$ and~$b$ be non-negative integers. 
    Let~$L$ be an $(a,b)$-lollipop with 
    witness~${(S,z,C)}$ 
    and let~${F\in \mathcal{F}(L)|C}$. 
    Then~${L \ast F}$ is an ${(a,b)}$-lollipop with witness~${(S,z,C)}$. 
\end{corollary}

\begin{proof}
    Let~${G := G(L)}$ and~${G' := G(L \ast F)}$.  
    By Proposition~\ref{prop:twisted-minors}\ref{prop-item:feasible-twist}, there is a matroid~$M$ associated with both~$L$ and~${L \ast F}$. 
    Since~$G$ is connected, so is~$M$ by Proposition~\ref{prop:fundamental-graph}, and hence so is~$G'$. 

    By Proposition~\ref{prop:twisted-minors}\ref{prop-item:restricted-twist}, ${(L \ast F)[C]}$ is equal to~${L[C] \ast F}$, and hence has branch-depth at least~$b$ and is connected. 
    
    Now by Proposition~\ref{prop:fundamental-graph-outside-twist}, the neighbourhood of each~${s \in S} $ is the same in~$G$ and~$G'$. 
    Hence ${G[S \cup \{z\}] = G'[S \cup \{z\}]}$, and no~${s \in S}$ has a neighbour in~$C$ in~$G'$. 
    Hence~${G'[C]}$ is indeed a component of~${G' - z}$. 
\end{proof}

\subsection{The induction}\label{subsec:induction}

As mentioned in the previous subsection, we aim to prove Theorem~\ref{thm:lollipops} by induction on~$a$.
For the start of the induction we consider the following lemma. 

\begin{lemma}\label{lem:induction-start}
    Let~$b$ be a non-negative integer. 
    Every twisted matroid~$W$ of branch-depth at least~${b + 2}$ has a minor which is a~${(0,b)}$-lollipop.
\end{lemma}

\begin{proof}
    By Lemma~\ref{lem:component}, 
    $W$ has a component~$C$ such that~${W[C]}$ has branch-depth at least~${b+1}$. 
    Let~${z \in C}$ be arbitrary. 
    By Lemma~\ref{lem:Xcomponent}, ${W[C \setminus \{z\}]}$
    has a connected component~$C'$ of branch-depth at least~${b}$. 
    Now~${W[C' \cup \{z\}]}$ 
    is a $(0,b)$-lollipop witnessed 
    by~${(\emptyset, z, C')}$ 
    since~${G(W[C' \cup \{z\}])}$ is connected. 
\end{proof}

For the induction step, the following two lemmas are the main tools we will need. 

\begin{lemma}\label{lem:lollipop-extension}
    Let~$a$,~$b$, and~$b'$ be non-negative integers. 
    Let~$L$ be an ${(a, b)}$-lollipop with 
    witness~${(S,z,C)}$ 
    and let~${C' \subseteq C}$ be non-empty such that 
    \begin{enumerate}
        [label=(\arabic*)]
        \item ${L[C']}$ is connected and has branch-depth at least~${b'}$; and 
        \item the neighbourhood of~$z$ in~${G := G(L)}$ is disjoint from~$C'$. 
    \end{enumerate}
    Then there exist a set~${S' \supseteq S}$ and an element~${z' \in C \setminus C'}$ such that~${L[S' \cup \{z'\} \cup C']}$ is an ${(a+1,b')}$-lollipop with witness~${(S', z', C')}$. 
\end{lemma}

\begin{proof}
    There is a shortest path~$P$ from~$z$ to~$C'$ in~${G[C \cup \{z\}]}$. 
    Let~$x$ be the unique vertex in~${V(P) \cap C'}$, let~$z'$ be the neighbour of~$x$ in~$P$, and let~${S' := S \cup (V(P) \setminus \{x,z'\})}$. 
    Now ${\abs{S'} \geq \abs{S} + 1 \geq a + 1}$,
    since~$z$ has no neighbour in~$C'$. 
    Hence ${L' := L[S'\cup \{z'\}\cup C']}$ is an ${(a + 1, b')}$-lollipop 
    witnessed by ${(S', z', C')}$, as desired. 
\end{proof}

If we can iteratively find lollipops in the candies of previously chosen lollipops, the next lemma will allow us to find a feasible set $F$ such that twisting by $F$ `displays' all the lollipops at the same time. 

\begin{lemma}\label{lem:nested-lollipops}
    Let~$\ell$ be a positive integer and let~$a$ and $g_\ell$ be non-negative integers. 
    Let~$W$ be a twisted matroid and let ${(g_i \colon 0 \leq i <\ell)}$ be a family of integers such that
    \begin{enumerate}
        [label=(\arabic*)]
        \item $W$ has branch-depth at least~$g_0$; 
        \item for all~${i < \ell}$, every minor of~$W$ of branch-depth at least~$g_i$ contains an ${(a,g_{i+1})}$-lollipop as a minor.
    \end{enumerate}
    Then there is a feasible set~$F$ and for each~${i \in [\ell]}$ there is a 
    set ${E_i = S_i \dot\cup \{z_i\} \dot\cup C_i}$ 
    such that 
    for~${W' := W \ast F}$ 
    the following properties hold.
    \begin{enumerate}
        [label=(\roman*)]
        \item ${L_i := W'[E_i]}$ is an ${(a,g_i)}$-lollipop witnessed 
        by~${(S_i,z_i,C_i)}$ for all~${i \in [\ell]}$; and
        \item ${E_{i+1} \subseteq C_i}$ for all~${i \in [\ell - 1]}$. 
    \end{enumerate}
\end{lemma}

\begin{proof}
    Let~${W_0 := W}$ and let~${C_0 := E(W)}$. 
    For~${i \in [\ell]}$ let~$L'_i$ be an ${(a,g_i)}$-lollipop with candy~$W_i$ such that~$L'_i$ is a minor of~$W_{i-1}$.
    Note that~$L'_i$ exists by the premise of the lemma. 
    Let~${(S_i,z_i,C_i)}$ be the witness of~$L'_i$ and let~${E_i := S_i \cup \{z_i\} \cup C_i}$.

    Let~${F'_0 := \emptyset}$. 
    For~${i \in [\ell]}$, let ${F_{i-1} \in \mathcal{F}(W_{i-1})}$ be such that ${L'_i = (W_{i-1} \ast F_{i-1})[E_i]}$, 
    and recursively define~${F'_i := F'_{i-1} \triangle F_{i-1}}$. 
    We now prove the following.
    \begin{claim*}\label{clm:nested-lollipop-claim}
        For~${i \in [\ell]}$, we have
        \begin{enumerate}
            [label=(\alph*)]
            \item\label{item:F'-feasibility} ${F'_i \in \mathcal{F}(W)}$,
            \item\label{item:big-triangle} ${L'_i = (W \ast F'_i)[E_i]}$, and
            \item\label{item:F-F'-containment} ${F'_{\ell-i} \triangle F'_\ell \subseteq C_{\ell-i}}$.
        \end{enumerate}
    \end{claim*}
    
    \begin{proof}[Proof of Claim]
        For ${i = 1}$, \ref{item:F'-feasibility} and~\ref{item:big-triangle} follow from the fact that~${F'_1 = F_0}$. 
        For~${i > 1}$, assume inductively that~${F'_{i-1} \in \mathcal{F}(W)}$ and~${L'_{i-1} = (W \ast F'_{i-1})[E_{i-1}]}$. 
        Since~${W_{i-1} = L'_{i-1}[C_{i-1}]}$, we have~${L'_{i} = (L'_{i-1} \ast F_{i-1})[E_i]}$.
        Therefore~\ref{item:F'-feasibility} and~\ref{item:big-triangle} follow from Proposition~\ref{prop:twist-transitivity}. 
        
        For~${i = 1}$,~\ref{item:F-F'-containment} follows from the fact that~${F'_\ell=F'_{\ell-1}\triangle F_{\ell-1}}$, and so~${F'_{\ell-1}\triangle F'_\ell=F_{\ell-1}}$. 
        For~${i > 1}$, assume by induction that~${F'_{\ell-(i-1)} \triangle F'_\ell \subseteq C_{\ell-(i-1)}}$. 
        Since~${F'_{\ell-(i-1)} = F'_{\ell-i} \triangle F_{\ell-i}}$, we have~${F'_{\ell-i} = F'_{\ell-(i-1)} \triangle F_{\ell-i}}$ and hence~$F'_{\ell-i} \triangle F'_\ell = (F'_{\ell-(i-1)} \triangle F'_{\ell}) \triangle F_{\ell-i}$. 
        Hence,~\ref{item:F-F'-containment} follows from the inductive hypothesis and the fact that both~$C_{\ell-(i-1)}$ and~$F_{\ell-i}$ are subsets of~$C_{\ell-i}$. 
    \end{proof}
    
    Define~${F := F'_\ell}$. 
    For~${i \in [\ell]}$, we have ${F \triangle F'_i \in \mathcal{F}(W \ast F)}$ by~\ref{item:F'-feasibility}. Therefore, by~\ref{item:big-triangle} and~\ref{item:F-F'-containment}, we have 
    \begin{align*}
        L'_i 
        &= (W \ast F'_i)[E_i]\\
        &= (W \ast (F'_i \triangle (F \triangle F)))[E_i]\\
        &= ((W \ast F)[E_i]) \ast (F \triangle F'_i).
    \end{align*}
    Hence, by Corollary~\ref{cor:lollipop-twist}, ${L_i := (W \ast F)[E_i]}$ is an ${(a,g_i)}$-lollipop witnessed by~${(S_i,z_i,C_i)}$, as required.
\end{proof}

Combining these two lemmas will be the heart of the induction step, as noted in the following corollary of the previous two lemmas. 

\begin{corollary}\label{cor:induction-step}
    In the situation of Lemma~\ref{lem:nested-lollipops}, additionally let~$b$ be a non-negative integer and assume that 
    \begin{enumerate}
        [label=$(\ast)$]
        \item\label{item:step-assumption} 
            there is an~${i \in [\ell]}$, a set~${C \subseteq C_\ell\subseteq C_i}$, and a feasible set~${\hat{F} \in \mathcal{F}(L_i)|C_i}$ 
            such that 
            \begin{enumerate}
                [label=(\arabic*)]
                \item $(L_i \ast \hat{F})[C]$ is connected and has branch-depth at least~$b$; and
                \item the neighbourhood of~$z$ in $G(L_i \ast \hat{F})$ is disjoint from~$C$.
            \end{enumerate}
    \end{enumerate}
    Then~$W$ contains an ${(a+1,b)}$-lollipop as a minor. \qed
\end{corollary}

Up to this point, we have not used the fact that the twisted matroid has bounded branch-width. 
In the next subsection we will prove the following lemma, which will complete the proof of Theorem~\ref{thm:lollipops}. 

\begin{restatable}{lemma}{step}\label{lem:induction-step-possible}
    Let~${b \geq 0}$ and~${w > 2}$ be integers. 
    Suppose we are in the situation of Lemma~\ref{lem:nested-lollipops} with ${\ell := 3 w -2}$ and~${g_\ell := b + 2w -1}$. 
    If~$W$ has branch-width at most~$w$, then we satisfy assumption~\ref{item:step-assumption} from Corollary~\ref{cor:induction-step}. 
\end{restatable}

\subsection{Proof of Lemma~\ref{lem:induction-step-possible}}
\label{subsec:morelemmas}

The following two lemmas are the final tools we will need for this proof.

\begin{lemma}\label{lem:lollipop-bridge}
    Let~${w > 2}$,~${k> 0}$, and~${b \geq 0}$ be integers, let~$W$ be a twisted matroid of branch-width at most~$w$ and let~$Z$ and~$C$ be disjoint subsets of~${E(W)}$ 
    such that~${\abs{Z} \geq 3k + 1}$ and~$W[C]$ has branch-depth at least~${b + w -1}$. 
    Then for some~${X \subseteq Z}$ and~${Y \subseteq C}$, the following hold. 
    \begin{enumerate}
        [label=(\roman*)]
        \item ${\abs{X} \geq k + 1}$;
        \item ${W[Y]}$ is connected and has branch-depth at least~$b$;
        \item ${\lambda_{W[X \cup Y]}(X) < w}$.
    \end{enumerate}
\end{lemma}

\begin{proof}
    Since~$W$ has branch-width at most~$w$, so does~${W' := W[Z \cup C]}$ by Lemma~\ref{lem:branch-width-minor-closed}. 
    Therefore, by Lemma~\ref{lem:AXYseparation}, 
    there exists a bipartition~${(X',Y')}$ of~${E(W')}$ with~${\lambda_{W'}(X') < w}$ 
    such that ${\abs{Z \cap X'} > k}$ 
    and~${\abs{Z \cap Y'} > k}$. 
    Now we can apply Lemma~\ref{lem:XYcomponents} to the matroids associated with~${W[X' \cap C]}$ and~${W[Y' \cap C]}$ as obtained from Proposition~\ref{prop:twisted-minor-associated}\ref{prop-item:twisted-minor-associated1}. 
    Without loss of generality~$W[Y' \cap C]$ 
    has a component~$Y$ of branch-depth at least~$b$. 
    Let~${X := X' \cap Z}$. 
    Since~${\lambda_{W'}(X') < w}$, it follows from Lemma~\ref{lem:connfunction} 
    that~${\lambda_{W[X \cup Y]}(X) < w}$. 
\end{proof}

\begin{lemma}\label{lem:feasibility}
    Let~${w > 2}$ be an integer and let~$W$ be a twisted matroid. 
    Then for every bipartition~${(X,Y)}$ of~${E(W)}$ with ${\abs{X} \geq w }$ and~${\lambda_W(X) < w}$ there is a base $B$ of $W$ and a set~${O \subseteq X \setminus B}$ of size at most~${w}$ such that~$O$ is a circuit in the matroid~${(W \ast B)\slash (B\cap X)}$.
\end{lemma}

\begin{proof}
    Let~$B_1$ be a base of~$W$. 
    We set~${B_2 := E(W) \setminus B_1}$, as well as ${M_1 := W \ast B_1}$ and ${M_2 := W \ast B_2}$. 
    Note that~${M_1^* = M_2}$. 
    Now we observe that 
    \begin{align*}
        \lambda_{W}(X) 
            &= \lambda_{M_1}(X) 
            = r_{M_1}(X) + r_{M_2}(X) - \abs{X}\\ 
            &= (r_{M_1}(X) - \abs{X \cap B_1}) + (r_{M_2}(X) - \abs{X \cap B_2})\\ 
            &= r_{M_1 \slash (B_1 \cap X)}(X \setminus B_1) + r_{M_2 \slash (X \cap B_2)}(X \setminus B_2). 
    \end{align*}
    Hence~${r_{M_1 \slash (B_1 \cap X)}(X \setminus B_1) + r_{M_2 \slash (X\cap B_2)}(X \setminus B_2) < w}$. 
    Since~${\abs{X} \geq w}$, for some base ${B \in \{ B_1, B_2 \}}$ and for~${M := W \ast B}$, 
    we have that~${\abs{X \setminus B} \geq r_{M \slash (B \cap X)}(X \setminus B) + 1}$.
    It follows that~${X \setminus B}$ contains a circuit~$O$ of size at most~${w}$ in~${M \slash (B \cap X)}$. 
\end{proof}


\begin{proof}[Proof of Lemma~\ref{lem:induction-step-possible} ]
    By applying Lemma~\ref{lem:lollipop-bridge} to~$W'$,~${Z := \{ z_i \colon i \in [\ell] \}}$, and~$C_\ell$, there are sets~${X \subseteq Z}$ and~${Y \subseteq C_\ell}$ such that 
    \begin{enumerate}
        [label=(\roman*)]
        \item\label{item:size-X} ${\abs{X} \geq w }$;
        \item\label{item:branch-depth-Y} ${W'[Y]}$ is connected and has branch-depth at least~${b+w}$;
        \item\label{item:connectivity-X} ${\lambda_{W'[X \cup Y]}(X) < w}$.
    \end{enumerate}

    By applying Lemma~\ref{lem:feasibility} to~${W'' := W'[X\cup Y]}$
    there exists a base~$B$ of $W''$ and a set~${O \subseteq X \setminus B}$ of size at most~${w}$ such that~$O$ is a circuit in~${(W'' \ast B) \slash (X \cap B)}$. 
    Now, for the restriction of that matroid to~${O \cup Y}$, which we call~$M$, it follows that~$O$ is a circuit of~$M$. 
    Note that since ${B \setminus E(M) = B \cap X}$, we get~${M = (W''[E(M)]) \ast (B \cap Y)}$ by Proposition~\ref{prop:twisted-minors}\ref{prop-item:compatible-minor}. 
    Since~${M|Y = W'[Y] \ast (B \cap Y)}$, it follows from~\ref{item:branch-depth-Y} that~$M$ has branch-depth at least~${b+w}$. 
    
    Let~${i \in [\ell]}$ be minimal such that~${z_i \in O}$. 
    Then we obtain~${O \setminus \{z_i\} \subseteq C_i}$. 
    Let~$\hat{B}$ be a base of~$M$ such that~${z_i \notin \hat{B}}$ and~$O$ is the fundamental circuit of~$z_i$ with respect to~$\hat{B}$. 
    
    Since the branch-depth of~${M \ast \hat{B}}$ equals the branch-depth of~$M$, by Lemma~\ref{lem:Xcomponent} 
    there is a component~$C$ of~${(M \ast \hat{B})[E(M) \setminus O]}$ of branch-depth at least~$b$. 
    
    Now ${\hat{F} := (B \cap E(M)) \triangle \hat{B}}$ is feasible with respect to~${W''[E(M)]}$ by Proposition~\ref{prop:twisted1}\ref{prop-item:feasible}, and since~${z_i \notin \hat{F}}$ we get~${\hat{F} \in \mathcal{F}(L_i)|C_i}$. 
    By Propositions~\ref{prop:twisted1}\ref{prop-item:fundamental-graph} 
    and~\ref{prop:twisted-minors}\ref{prop-item:restriction}, 
    \[
        {G(M, \hat{B}) = G(W'[E(M)] \ast \hat{F}) = G(L_i \ast \hat{F})[E(M)]}.
    \] 
    Hence, by our choice of~$\hat{B}$ the neighbourhood of~$z_i$ in~$G(L_i \ast \hat{F})$ is~${O \setminus z_i}$, which is disjoint from~$C$. 
    And since $(L_i \ast \hat{F})[C] = (M \ast \hat{B})[C]$, we obtain condition~\ref{item:step-assumption} of Corollary~\ref{cor:induction-step}, as desired. 
\end{proof}

\subsection{Proof of Theorem~\ref{thm:lollipops}}\label{subsec:proof}

We now prove Theorem~\ref{thm:lollipops}, which completes the proof of Theorem~\ref{thm:main-intro}. 

\begin{definition}\label{def:function}
    Given non-negative integers~$a$,~$b$,~$w$ and~$i$ with~${w \geq 2}$, we make the following definition, with~${\ell := 3w-2}$.
    \[
        g(a,b,w,i) := 
        b+(2w-1) + \left( (2w-1) \frac{(\ell^{a-1}-1)}{\ell-1} + 2\ell^{a-1} \right) (\ell-i).
    \]
    
    For each integer~${w > 2}$, 
    we define a function~${f_w \colon \mathbb{N}^2 \to \mathbb{N}}$ 
    by~${f_w(a,b) := g(a,b,w,0)}$.
\end{definition}

\begin{lemma}\label{lem:recursive-function}
    If~$a$,~$b$,~$w$, and~$i$ are non-negative integers with~${w \geq 2}$ and~${i < 3w-2}$, then
    \begin{enumerate}
        [label=(\roman*)]
        \item \label{lem-item:rec1} ${f_w(0,b) = b+2}$,
        \item \label{lem-item:rec2} ${g(a,b,w,3w-2) = b+2w-1}$,
        \item \label{lem-item:rec3} ${g(a+1,b,w,i) = f_w(a,g(a+1,b,w,i+1))}$, and 
        \item \label{lem-item:rec4} ${f_w(a,b) \leq b+3(3w-2)^a}$. 
    \end{enumerate}
\end{lemma}
\begin{proof}
    Define~${\ell := 3w - 2}$. 
    For~\ref{lem-item:rec1}, 
    \begin{align*}
        f_w(0,b) 
        &= g(0,b,w,0)
        = b+(2w-1)+\left( (2w-1) \frac{\ell^{-1}-1}{\ell-1} + 2\ell^{-1} \right)\ell\\
        &= b+(2w-1)+\left( (2w-1)\frac{1-\ell}{\ell-1}+2\right)=b+2.
    \end{align*}
    Property~\ref{lem-item:rec2} is trivial. 
    For~\ref{lem-item:rec3}, 
    \begin{align*}
        \lefteqn{f_w(a,g(a+1,b,w,i+1))}\\
        &= g(a+1,b,w,i+1)+(2w-1)+\left( (2w-1) \frac{\ell^{a-1}-1}{\ell-1} + 2\ell^{a-1} \right)\ell\\
        &= \left( g(a+1,b,w,i)-\left( (2w-1) \frac{\ell^{a}-1}{\ell-1} + 2\ell^{a} \right) \right) + (2w-1)\frac{(\ell-1)+(\ell^{a}-\ell)}{\ell-1}+2\ell^{a}\\
        &= g(a+1,b,w,i).
    \end{align*}
    For~\ref{lem-item:rec4}, note that~${2w-1 \leq \ell-1}$, and hence
    \begin{align*}
        b + (2w-1) + \left( (2w-1) \frac{\ell^{a-1}-1}{\ell-1} + 2\ell^{a-1} \right) \ell 
        &\leq b + \ell + (3\ell^{a-1}-1) \ell 
        = b + 3(3w-2)^a.
        \qedhere
    \end{align*}
\end{proof}

For convenience, we restate Theorem~\ref{thm:lollipops}.

\lollipops*

\begin{proof}
    We may assume without loss of generality that~$w$ is at least~$3$. 
    We will prove by induction on~$a$ that the theorem holds with~${d := f_w(a,b)}$ as in Definition~\ref{def:function}. 
    The base case is true by Lemma~\ref{lem:induction-start} and~Lemma~\ref{lem:recursive-function}\ref{lem-item:rec1}. 
    For the induction step, let~${\ell:=3w-2}$, and note that the premise of Lemma~\ref{lem:nested-lollipops} holds with the family~${(g_i \colon 0 \leq i \leq \ell)}$, where~${g_i}$ is defined as~${g(a+1,b,w,i)}$ as in Definition~\ref{def:function} for all~${i \in [\ell]}$, by Lemma~\ref{lem:branch-width-minor-closed}, Lemma~\ref{lem:recursive-function}, and the induction hypothesis. 
    Hence, Lemma~\ref{lem:induction-step-possible} together with Corollary~\ref{cor:induction-step} completes the proof. 
\end{proof}

\section{Consequences}\label{sec:consequences}

\subsection{Matroids representable over a fixed finite field}

Now we can prove Corollary~\ref{cor:main-intro-representable}, 
which we restate for the convenience of the reader. 

\mainthmintrorepresentable*

Since neither~$U_{2,q+2}$ nor~$U_{q,q+2}$ is representable over~$\field(q)$, 
we will instead show the following stronger corollary, implying Corollary~\ref{cor:main-intro-representable}.

\begin{corollary}\label{cor:main-intro-representable2}
    For any positive integers~$n$ and~$q$, 
    there is an integer~$d$
    such that every matroid having no minor isomorphic to~$U_{2,q+2}$ or~$U_{q,q+2}$
    with branch-depth at least~$d$ 
    contains a minor isomorphic to~${M(F_n)}$. 
\end{corollary}

The \emph{${m \times n}$ grid} is the graph with vertices~${\{ (i,j) \colon i \in [m], j \in [n]\}}$, where~${(i,j)}$ and~${(i',j')}$ are adjacent if and only if~${\abs{i-i'} + \abs{j-j'} = 1}$.
The above corollary is obtained by using 
the following theorem of Geelen, Gerards, and Whittle~\cite{GGW:large-branchwidth}, because 
the cycle matroid of the~${n \times n}$ grid contains~${M(F_n)}$ as a minor.

\begin{theorem}[Geelen, Gerards, and Whittle~\cite{GGW:large-branchwidth}*{Theorem~2.2}]\label{thm:grid-theorem2}
    For any positive integers $n$ and~$q$,
    there is an integer~${w(n,q)}$ 
    such that every matroid having no minor isomorphic to~$U_{2,q+2}$ or~$U_{q,q+2}$
    with branch-width at least~${w(n,q)}$     
    contains a minor isomorphic to the cycle matroid of the~${n \times n}$ grid. 
\end{theorem}

\begin{proof}[Proof of Corollary~\ref{cor:main-intro-representable2}]
    Let~${w := w(n,q)}$ given by Theorem~\ref{thm:grid-theorem2}. 
    Let~$d$ be the integer given by Theorem~\ref{thm:main-intro} for~$n$ and~$w$. 
    Since the cycle matroid of the ${n \times n}$ grid contains~${M(F_n)}$ as a minor, 
    we deduce the conclusion easily.
\end{proof}

\subsection{Quasi-graphic matroids}

Geelen, Gerards, and Whittle~\cite{quasi-graphic1} introduced 
the class of quasi-graphic matroids, 
which includes the classes of 
graphic matroids, 
bicircular matroids, 
frame matroids, 
and lift matroids. 
We will show that quasi-graphic matroids of large branch-depth contain large fan minors, as a corollary of Theorem~\ref{thm:main-intro}. 

Though the original definition 
of quasi-graphic matroids is due to Geelen, Gerards, and Whittle~\cite{quasi-graphic1}, we present the equivalent definition of Bowler, Funk, and Slilaty~\cite{quasi-graphic2}. 
Let~$G$ be a graph. 
A tripartition~${(\mathcal{B}, \mathcal{L}, \mathcal{F})}$ of cycles of~$G$ into possibly empty sets 
is called \emph{proper} if it satisfies the following properties. 
\begin{enumerate}
    [label=(\roman*)]
    \item $\mathcal{B}$ satisfies the \emph{theta property}: 
        if~$C_1$,~$C_2$ are two cycles in~$\mathcal{B}$ such that~${E(C_1) \triangle E(C_2)}$ is the edge set of a cycle~$C$, 
        then~$C$ is in~$\mathcal{B}$. 
    \item Whenever~$L$ is in~$\mathcal{L}$ and~$F$ is in~$\mathcal{F}$, there is at least one common vertex of~$L$ and~$F$. 
\end{enumerate}
A cycle is \emph{balanced} if it is in~$\mathcal{B}$
and \emph{unbalanced} otherwise. 
Let~$X$ be a subset of~${E(G)}$. 
If the subgraph~${G[X]}$ contains no unbalanced cycle, then we say that~$X$ and~${G[X]}$ are \emph{balanced}, and otherwise we say they are \emph{unbalanced}. 
A \emph{theta graph} is a subgraph consisting of three internally disjoint paths joining two distinct vertices.

For a graph~$G$ and a proper tripartition~$(\mathcal{B},\mathcal{L},\mathcal{F})$ of its cycles, 
we define a matroid ${M = M(G, \mathcal{B}, \mathcal{L}, \mathcal{F})}$ by describing its circuits as follows. 
A subset~$X$ of~${E(G)}$ is a circuit of~$M$ if and only if~$X$ is the edge set of one of the following: 
\begin{enumerate}
    [label=(\arabic*)]
    \item a balanced cycle, 
    \item a theta graph containing no balanced cycle, 
    \item the union of two edge-disjoint unbalanced cycles sharing exactly one vertex 
        (such a subgraph is called a \emph{tight handcuff}),  
    \item the union of two vertex-disjoint cycles in~$\mathcal{L}$, and
    \item the union of two vertex-disjoint cycles in~$\mathcal{F}$ and a minimal path joining these two cycles 
        (such a subgraph is called a \emph{loose handcuff}). 
\end{enumerate}

If~$\mathcal{B}$ contains every cycle of~$G$, then ${M = M(G)}$ is a graphic matroid. 
If~$\mathcal{L}$ is empty, then~$M$ is a \emph{frame} matroid, 
and if~$\mathcal{F}$ is empty, then~$M$ is a \emph{lift} matroid. 
If both~$\mathcal{B}$ and~$\mathcal{L}$ are empty, then~$M$ is a \emph{bicircular} matroid. 

\begin{proposition}\label{prop:branch-width-graphs-quasi-graphic}
    Let~$G$ be a graph, and let~${(\mathcal{B},\mathcal{L},\mathcal{F})}$ be a proper tripartition of the cycles of~$G$. 
    If~$G$ has branch-width at most~$w$, then the quasi-graphic matroid ${M := M(G,\mathcal{B},\mathcal{L},\mathcal{F})}$ has branch-width at most~${w+3 }$. 
\end{proposition}

\begin{proof}
    We may assume that~$G$ has at least~$2$ edges. 
    Let~${(T,\sigma)} $ be a branch-decomposition of the graph~$G$ with width at most~$w$. 
    This means that whenever~$e$ is an edge of~$T$, there are at most~$w$ vertices incident with both sides of a partition~${(A_{e},B_{e})}$ of~${E(G)}$ induced by the components of~${T - e}$ under~$L^{-1}$. 
    We will demonstrate that~${\lambda_{M}(A_{e}) \leq w+2 }$ for every edge~$e$, and then~${(T,\sigma)} $ will certify the branch-width of~$M$ to be at most~${w+3 }$. 
    
    Let~$X$ be a subset of~${E(G)}$. 
    Let~${c(X)}$ denote the number of connected components in the subgraph~${G[X]}$, and let~${b(X)}$ denote the number of these components that are balanced. 
    Moreover, let~${\ell(X)}$ be~$1$ if~${G[X]}$ contains a cycle in~$\mathcal{L}$, and otherwise let~${\ell(X)}$ be~$0$. 
    The rank~${r_{M}(X)}$ is given by the formula~${\abs{V(X)} - b(X)}$ when~${G[X]}$ contains a cycle in~$\mathcal{F}$, and otherwise by~${\abs{V(X)} - c(X) +\ell(X)}$ \cite{quasi-graphic2}*{Lemma~2.4}. 
    
    Let~$n$ be the number of vertices in~$G$ and let~${E := E(G)}$. 
    Let~$n_{A}$ and~$n_{B}$ be the number of vertices in the subgraphs~${G[A_{e}]}$ and~${G[B_{e}]}$, so that~${n_{A} + n_{B} - n}$ is the number of vertices incident both with edges in~$A_{e}$ and edges in~$B_{e}$. 
    
    First assume that~$\mathcal{F}$ is non-empty. 
    Then~${r(M) = n - b(E)}$.
    We split into three subcases depending on whether both, one, or neither of~${G[A_{e}]}$ and~${G[B_{e}]}$ contain cycles in~$\mathcal{F}$. 
    Assume that both~${G[A_{e}]}$ and~${G[B_{e}]}$ contain cycles in~$\mathcal{F}$. 
    Any subgraph of a balanced subgraph is itself balanced, and it follows that~${b(A_{e}) + b(B_{e}) \geq b(E)}$. 
    Therefore
    \begin{align*}
        \lambda_{M}(A_{e}) 
        &= r_{M}(A_{e}) + r_{M}(B_{e}) - r(M) 
        = (n_{A} - b(A_{e})) + (n_{B} - b(B_{e})) - (n - b(E))
        \leq \abs{V(A_{e}) \cap V(B_{e})} \leq w    
    \end{align*}
    as desired. 
    Now assume that~${G[A_{e}]}$ contains a cycle in~$\mathcal{F}$ but that~${G[B_{e}]}$ does not. 
    In this case ${b(A_{e})+c(B_{e})\geq b(E)}$, so
    \[
        \lambda_{M}(A_{e}) 
        = (n_{A} - b(A_{e})) + (n_{B} - c(B_{e}) +\ell(B_{e})) - (n - b(E)) \leq \abs{V(A_{e}) \cap V(B_{e})}+\ell(B_{e})\leq w +1 . 
    \]
    If neither~${G[A_{e}]}$ nor~${G[B_{e}]}$ contains a cycle in~${\mathcal{F}}$, then since~${c(A_{e}) + c(B_{e}) \geq c(E) \geq b(E)}$, 
    we conclude that~${\lambda_{M}(A_{e}) \leq w+\ell(A_e)+\ell(B_e)\leq w+2}$. 
    
    Now we assume that~$\mathcal{F}$ is empty. Therefore 
    \[
        {r(M) = n - c(E) +\ell(E)},~ 
        {r(A_{e}) = n_{A} - c(A_{e}) +\ell(A_{e})}, 
        \textnormal{ and }
        {r(B_{e}) = n_{B} - c(B_{e}) +\ell(B_{e})}. 
    \]
    As~${c(A_{e}) + c(B_{e}) \geq c(E)}$, it follows easily that~${\lambda_M(A_{e}) \leq w + 2}$, and this completes the proof. 
\end{proof}

We will use the following grid theorem due to Robertson and Seymour. 
Note that in~\cite{graphminorsV} they proved this theorem in terms of tree-width, but in~\cite{graphminorsX} they established that graphs have small tree-width if and only if they have small branch-width, yielding the following version of the theorem. 

\begin{theorem}[Robertson and Seymour~\cite{graphminorsV}*{(2.1)} and \cite{graphminorsX}*{(5.1)}]\label{thm:grid}
    For any positive integer~$n$, there is an integer~$N(n)$ such that every graph of branch-width at least~${N(n)}$ contains a minor isomorphic to the~${n \times n}$ grid.
\end{theorem}

For a positive integer~$n$, let~$P_n^\circ$ be the graph obtained from the path on~$n$ vertices by adding one loop at each vertex.
By comparing circuits, it is easy to observe the following lemma.
\begin{lemma}\label{lem:bicircular-fan}
    For every positive integer~$n$, the bicircular matroid $M(P_n^\circ,\emptyset,\emptyset,\mathcal{C}_n^\circ)$ is isomorphic to~$M(F_n)$, where~$\mathcal{C}_n^\circ$ is the set of cycles of~$P_n^\circ$. \qed
\end{lemma}

\begin{proposition}\label{prop:quasi-branch-width}
    For every positive integer~${n}$,
    there is an integer~${w}$ 
    such that every quasi-graphic matroid with branch-width at least~$w$ 
    contains a minor isomorphic to~${M(F_n)}$. 
\end{proposition}

\begin{proof}
    We may assume that~${n > 2}$. 
    Let~${w := N(n^2) + 3}$ where $N(n^2)$ is the integer given in Theorem~\ref{thm:grid}.

    Let~${M = M(G,\mathcal{B},\mathcal{L},\mathcal{F})}$
    be a quasi-graphic matroid of branch-width at least~$w$. 
    Assume for a contradiction that~$M$ does not have a minor isomorphic to~${M(F_{n})}$. 
    
    By Proposition~\ref{prop:branch-width-graphs-quasi-graphic}, $G$ has branch-width at least~$N(n^2)$. 
    By Theorem~\ref{thm:grid}, $G$ has a minor~$G'$ isomorphic to the ${n^2 \times n^2}$ grid. 
    We may assume that $G'$ is equal to the ${n^2 \times n^2}$ grid. 
    
    As~$G'$ is obtained from~$G$ by deleting edges and contracting non-loop edges, it follows immediately from \cite{quasi-graphic2}*{Theorem~4.5} that there is a proper tripartition~${(\mathcal{B}', \mathcal{L}', \mathcal{F}')}$ of the cycles of~$G'$ such that~${M' := M(G', \mathcal{B}', \mathcal{L}', \mathcal{F}')}$ is a minor of~$M$. 
    
    First assume~$\mathcal{L}'$ contains a cycle of length~$4$ with edge set~${\{c_{1},c_{2},c_{3},c_{4}\}}$. 
    By~\cite{quasi-graphic2}*{Theorem~4.5}, ${M'' := M' \slash \{c_{1},c_{2},c_{3}\}}$ is a quasi-graphic matroid. 
    Let ${G'' := G' \slash \{c_{1},c_{2},c_{3}\}}$, and let~${(\mathcal{B}'',\mathcal{L}'',\mathcal{F}'')}$ be a proper tripartition of the cycles of~$G''$ so that~${M'' = M(G'',\mathcal{B}'',\mathcal{L}'',\mathcal{F}'')}$. 
    Again by \cite{quasi-graphic2}*{Theorem~4.5} we may assume that the cycle with edge set~$\{c_4\}$ is in~$\mathcal{L}''$. 
    Let~$v$ be the vertex of $G''$ that is incident with~$c_{4}$. 
    By definition, every cycle in~$\mathcal{F}''$ contains~$v$. 
    This means that~${M'' = M(G'',\mathcal{B}'', \mathcal{L}'' \cup \mathcal{F}'', \emptyset)}$ (see~\cite{quasi-graphic2}*{Section~2.3}). 
    For~${S \subseteq E(M'' \slash c_4)}$, the set~${S \cup \{c_4\}}$ is dependent in~$M''$ if and only if~$S$ contains the edge set of some cycle of~${G'' \slash c_4}$, so~${M''\slash c_4 = M(G'' \slash c_4) = M(G' \slash \{c_{1},c_{2},c_{3},c_{4}\})}$. 
    As~$n^2$ is greater than four, it follows that for some~${\alpha \in [n^2]}$ the graph $G''[\{(i,j) \colon i \in \{\alpha, \alpha+1\},~j \in [n^2]\}]$ is a subgraph of~${G'' \slash c_{4}}$ isomorphic to the ${2 \times n^2}$ grid. 
    By contracting the edges in the path~${(\alpha,1)(\alpha,2)\cdots (\alpha,n^2)}$, we obtain a minor isomorphic to~$F_{n^2}$. 
    Now~${n^2 > n}$ implies that~$M$ has a minor isomorphic to~${M(F_{n})}$, a contradiction. 
    Therefore~$\mathcal{L}'$ contains no cycle of length~$4$. 
    
    Consider the subgraph~${G_1 := G'[\{(i,j) \colon i \in \{1,2\},~j \in [n^2] \}]}$. 
    If~$G_1$ contains~$n$ vertex-disjoint cycles of length~$4$ in $\mathcal{F}'$, 
    then~$M'$ has a minor isomorphic to ${M(P_n^\circ,\emptyset, \emptyset, \mathcal{C}_n^\circ)}$, where~$\mathcal{C}_n^\circ$ is the set of cycles of $P_n^\circ$, contradicting our assumption by Lemma~\ref{lem:bicircular-fan}.
    
    Since the ${2 \times n}$ grid contains~$F_n$ as a minor, by our assumption, 
    any sequence of consecutive balanced cycles of length~$4$ in~$G_1$ contains at most~${n-2}$ such cycles.
    As~$G_1$ contains no cycles of length~$4$ in~$\mathcal{L}'$, and at most~$n-1$ 
    vertex-disjoint cycles of length~$4$ in~$\mathcal{F}'$, it follows that~$G_1$ contains at most~${(n - 2)n +  2(n-1) = n^{2} - 2}$ cycles of length~$4$.
    This is impossible, as~$G_1$ has at least~${n^2-1}$ cycles of length~$4$. 
\end{proof}

Now it is routine to combine Proposition~\ref{prop:quasi-branch-width}
with Theorem~\ref{thm:main-intro} to deduce the following result. 

\begin{corollary}\label{cor:quasigraphic}
    For every positive integer~${n}$, 
    there is an integer~${d}$ 
    such that every quasi-graphic matroid with branch-depth at least~$d$ 
    contains a minor isomorphic to~${M(F_n)}$. \qed
\end{corollary}

\subsection{General matroids}

The following conjecture about branch-width is due to Johnson, Robertson, and Seymour. 

\begin{conjecture}[Johnson, Robertson, and Seymour; see~\cite{excluding-uniform}*{Conjecture~6.1}]
    \label{conj:grid-theorem}
    For every positive integer~${n}$, 
    there is an integer~${d}$ 
    such that every matroid of branch-width at least~$d$ 
    contains a minor isomorphic to either 
    \begin{itemize}
        \item the cycle matroid of the ${n \times n}$ grid; 
        \item the bicircular matroid of the ${n \times n}$ grid; 
        \item the dual of the bicircular matroid of the ${n \times n}$ grid; or
        \item the uniform matroid~$U_{n,2n}$. 
    \end{itemize}
\end{conjecture}

Note that all matroids appearing in this conjecture have large branch-width, making them necessary obstructions. 

Hence, Theorem~\ref{thm:main-intro} yields the following corollary.

\begin{corollary}
    Conjecture~\ref{conj:grid-theorem} implies Conjecture~\ref{conj:branchdepth}. 
\end{corollary}

\begin{proof}
    
    Since both the cycle matroid and the bicircular matroid of the ${n \times n}$ grid are quasi-graphic and both have large branch-depth, by Corollary~\ref{cor:quasigraphic}, if a matroid~$M$ contains a minor isomorphic to one of them, it contains a minor of the fan matroid as well. 
    It follows that if a matroid~$M$ contains a minor isomorphic to the dual of the bicircular matroid of the ${n \times n}$ grid, then it contains the dual of the fan matroid.
    By Propositions~\ref{prop:path-to-fan} and~\ref{prop:fundamental-graph}\ref{rmk:dual-fundamental},~$M(F_n)^\ast$  contains a minor isomorphic to~$M(F_{n-1})$. 
\end{proof}


\section*{Acknowledgments} 
We wish to thank the reviewers for their many helpful suggestions. 
We are particularly grateful to the anonymous reviewer who took the time to determine the explicit bound in our main theorem from what was previously a somewhat opaque recursive definition.


%
%

\begin{aicauthors}
\begin{authorinfo}[jpg]
  J.~Pascal Gollin\\
  Discrete Mathematics Group\\
  Institute for Basic Science\\
  Daejeon, Korea\\
  pascalgollin\imageat{}ibs\imagedot{}re\imagedot{}kr \\
  \url{https://dimag.ibs.re.kr/home/gollin}
\end{authorinfo}
\begin{authorinfo}[kh]
  Kevin Hendrey\\
  Discrete Mathematics Group\\
  Institute for Basic Science\\
  Daejeon, Korea\\
  kevinhendrey\imageat{}ibs\imagedot{}re\imagedot{}kr \\
  \url{https://sites.google.com/view/kevinhendrey}
\end{authorinfo}
\begin{authorinfo}[dm]
  Dillon Mayhew\\
  School of Mathematics and Statistics\\
  Te Herenga Waka, 
  Victoria University of Wellington\\
  Wellington, New Zealand\\
  dillon\imagedot{}mayhew\imageat{}vuw\imagedot{}ac\imagedot{}nz\\
  \url{https://homepages.ecs.vuw.ac.nz/~mayhew/}
\end{authorinfo}
\begin{authorinfo}[so]
  Sang{-}il Oum\\
  \begin{tabular}{@{}p{0.4\textwidth}p{0.4\textwidth}}
  Discrete Mathematics Group & Department of Mathematical Sciences \\
  Institute for Basic Science & KAIST\\
  Daejeon, Korea & Daejeon, Korea
  \end{tabular}\\
  sangil\imageat{}ibs\imagedot{}re\imagedot{}kr \\
  \url{https://dimag.ibs.re.kr/home/sangil}
\end{authorinfo}
\end{aicauthors}

\end{document}